\documentclass[12pt,reqno]{amsart}
\usepackage{amssymb,amsmath,amsthm,amsfonts,times,graphicx}

\usepackage{color,setspace,wrapfig,multicol}
\usepackage[labelformat=empty]{subcaption}

\usepackage[left=27mm,right=27mm,top=21mm,bottom=21mm]{geometry}

\usepackage{comment}
\usepackage[colorlinks=true, pdfstartview=FitV, linkcolor=blue, citecolor=blue, urlcolor=blue]{hyperref}

\numberwithin{equation}{section}

\newcommand{\dispersion}{\mathfrak{e}}
\newcommand{\emax}{\dispersion_{\max}}
\newcommand{\emin}{\dispersion_{\min}}

\newcommand{\disc}{{\ensuremath{\mathrm{disc}}}}
\newcommand{\ess}{{\ensuremath{\mathrm{ess}}}}

\newcommand{\p}{\partial}

\newcommand{\os}{\mathrm{os}}
\newcommand{\oa}{\mathrm{oa}}
\newcommand{\ea}{\mathrm{ea}}
\newcommand{\es}{\mathrm{es}}

\renewcommand{\epsilon}{\varepsilon}
\renewcommand{\hat}{\widehat}
\renewcommand{\tilde}{\widetilde}

\newcommand{\C}{\mathbb{C}}

\newcommand{\R}{\mathbb{R}}

\newcommand{\T}{\mathbb{T}}
\newcommand{\Z}{\mathbb{Z}}
\newcommand{\cF}{{\mathcal F}}
\newcommand{\cH}{{\mathcal H}}
\renewcommand{\d}{\mathrm{d}}

{\bf}{\it}
\newtheorem{theorem}{Theorem}[section]
\newtheorem{lemma}[theorem]{Lemma}
\newtheorem{corollary}[theorem]{Corollary}
\newtheorem{hypothesis}[theorem]{Hypothesis}

\newtheorem{proposition}[theorem]{Proposition}
\newtheorem{remark}[theorem]{Remark}
\newtheorem{example}[theorem]{Example}

\date{\today}

\begin{document}
\title[]{On the spectrum of Schr\"odinger-type operators on two dimensional lattices}

\author[]{Shokhrukh Yu. Kholmatov, Saidakhmat N. Lakaev, Firdavsjon M. Almuratov}

\address[Sh. Kholmatov]{University of Vienna,
Oskar-Morgenstern-Platz 1, 1090  Vienna, Austria}
\email{shokhrukh.kholmatov@univie.ac.at}

\address[S. Lakaev]{Samarkand State University, University boulevard 15, 140104 Samarkand, Uzbekistan}
\email{slakaev@mail.ru}

\address[F. Almuratov]{Samarkand State University, University boulevard 15, 140104 Samarkand, Uzbekistan}
\email{almurotov93@mail.ru}

\begin{abstract}
We consider a  family 
$$
\widehat H_{a,b}(\mu)=\widehat  H_0 +\mu \widehat V_{a,b}\quad \mu>0,
$$ 
of Schr\"odinger-type operators on the two dimensional lattice $\Z^2,$  where $\widehat H_0$ is a Laurent-Toeplitz-type convolution operator with a given Hopping matrix $\hat \dispersion$ and  $\widehat V_{a,b}$ is a potential taking into account only the zero-range and one-range interactions, i.e.,  a multiplication operator by a function $\hat v$ such that $\hat v(0)=a,$ $\hat v(x)=b$ for $|x|=1$ and $\hat v(x)=0$ for $|x|\ge2,$ where $a,b\in\R\setminus\{0\}.$ Under certain conditions on the regularity of $\hat\dispersion$ we completely describe the discrete spectrum of $\hat H_{a,b}(\mu)$ lying above the essential spectrum and study the dependence of eigenvalues on parameters $\mu,$ $a$ and $b.$ Moreover, we characterize the threshold eigenfunctions and resonances.
\end{abstract}

\maketitle

\section{Introduction}

In the recent paper \cite{KhLA:2021_jmaa} we have studied the family
\begin{equation}\label{hamiltonian1}
\hat H_{\hat v}(\mu) = \hat H_0 + \mu\hat V,\quad\mu>0,
\end{equation}
of Schr\"odinger-type operators in $\ell^2(\Z^d)$ for $d=1,2,$
where $\hat H_0$ is a Laurent-Toeplitz-type convolution operator with a so-called Hopping matrix $\hat\dispersion,$ $\hat V$ is a multiplication operator by a function $\hat v$ and $\mu>0$ is a coupling constant. Under general assumptions on $\hat \dispersion$ and $\hat v,$ adapting the methods of \cite{K:1977_ann.phys} we were able to prove the finiteness of discrete spectrum of $\hat H_{\hat v}(\mu)$ (Bargmann-type estimates), to establish sufficient conditions for the existence of eigenvalues, to analyse existence and uniqueness of eigenvalues for small $\mu>0$ as well as to find the expansions of the largest and smallest eigenvalues of $\hat H_{\hat v}(\mu)$ as $\mu\searrow 0.$

The family $\hat H_{\hat v}(\mu)$ given by \eqref{hamiltonian1} can be seen as a discrete counterpart of a continuous nonlocal Schr\"odinger-type operators appearing
as a diffusion generator in certain stochastic models  (see e.g. \cite{KKP:2008,KMPZ:2016,KS:2006}). At the same time they can also be seen as an effective one-particle discrete
Schr\"odinger operator $\hat H_{a,b}(\mu)(K)$ parametrized by quasi-momenta $K\in\T^d$, associated to the Hamiltonian of a system of two arbitrary and/or identical particles  in $d$-dimensional lattice $\Z^d$ (see e.g.
\cite{ALMM:2006_cmp,FIC:2002_phys.rev, LKL:2012.tmf,LKh:2012_izv,LKh:2011_jpa}). Such lattice models have become popular in recent years because they represent a minimal, natural Hamiltonian describing the systems of ultracold atoms in optical lattices -- systems with highly controllable
parameters  such as lattice geometry and dimensionality, particle masses, tunneling, temperature etc.  (see e.g.,
\cite{B:2005_nat.phys,JBC:1998_phy.rev,JZ:2005_ann.phys,LSA:2012_book}
and references therein). In contrast to usual condensed matter systems, where stable composite objects are usually formed by
attractive forces and repulsive forces separate particles in free space, the  controllability of collision properties of ultracold atoms allowed to observe experimentally a stable repulsive bound pair of ultracold atoms in the optical
lattice $\Z^3$, see e.g.,
\cite{HWCR:2012,WTLG:2006_nature,ZNON:2008}.

In this paper we consider the potential
\begin{equation}\label{potential_fun0000}
\hat v(x)=
\begin{cases}
a, & \text{if $x=0,$}\\
b, & \text{if $|x|=1,$}\\
0, & \text{if $|x|>1,$}
\end{cases}
\quad a,b\in \R\setminus\{0\},
\end{equation}
in the two-dimensional lattice $\Z^2$ and study the discrete spectrum of the operator $\hat H_{a,b}(\mu):=\hat H_{\hat v}(\mu)$ lying above the essential spectrum. We also use the momentum representation $H_{a,b}(\mu)$ of $\hat H_{a,b}(\mu),$ given in the Hilbert space $L^2(\T^2)$ of square-integrable functions on the two dimensional torus $\T^2,$ via the Fourier transform $\cF$ as 
$$
H_{a,b}(\mu) = \cF\hat H_{a,b}(\mu)\cF^*
$$
(see \eqref{def:H_mu999} below).  We are interested in the quantitative and qualitative properties of  eigenvalues of $\hat H_{a,b}(\mu)$ depending on parameters $a,b\ne0$ and $\mu>0.$ More specifically, we study sufficient and necessary conditions for the existence of discrete spectrum $\sigma_\disc(\hat H_{a,b}(\mu_0))$, a characterization of emergence of eigenvalues and eigenfunctions, dependence of the number of eigenvalues on parameters $\mu$, $a$ and $b,$ and also the dependence of eigenvalues on $\mu$.

We recall that for (one-particle) Schr\"odinger operators $-\Delta + \mu V$ in  $\R^d$ the eigenvalue emergence is called \emph{coupling constant threshold phenomenon} -- a situation where for some $\mu_0\ge0,$  some eigenvalue $e(\mu)$ of $-\Delta + \mu V$ is absorbed by the essential spectrum, i.e., 
$e(\mu)\nearrow 0$ as $\mu\searrow \mu_0,$ and conversely, as  $\mu\nearrow \mu_0+\epsilon,$ the essential spectrum ``gives birth'' to a new eigenvalue $e(\mu)$ (see e.g., \cite{KS:1980} and the references therein). 
For lattice operators of type \eqref{hamiltonian1} the coupling constant threshold phenomenon has been studied for instance in \cite{LKh:2011_jpa,LKh:2012_izv,LKL:2012.tmf} for zero-range potential, i.e., for $\hat v$ in \eqref{potential_fun0000} with $b=0$. 

Note that by the Kato-Rellich perturbation theory  if $E(\mu_0)$ is an isolated eigenvalue of $\hat H_{\hat v}(\mu_0),$ then for all $\mu$ sufficiently close $\mu_0$ the operator $\hat H_{\hat v}(\mu)$ has an isolated eigenvalue $E(\mu)$ such that the function $\mu\mapsto E(\mu)$ is real analytic at $\mu_0.$ However, at a coupling constant threshold $\mu_0\ge0$, as an isolated eigenvalue,  $E(\mu)$ is absorbed by the essential spectrum, one does not expect analyticity, rather it is expected that $E(\mu)$ admits a convergent expansion in some singular terms containing $\mu-\mu_0.$ For very nice sign-definite potentials $V$ in $\R^d$ such expansions for eigenvalues of $-\Delta + \mu V$ have been obtained in all dimensions $d,$ for instance in \cite{H:1985_jot,KS:1980}, 
and for zero-range  potentials in $\Z^d$, for the eigenvalues of $\hat H_{\hat v}(\mu)$ analogous expansions in all dimensions were established in \cite{LKL:2012.tmf,LKh:2012_izv,LKh:2011_jpa}.  

The coupling constant threshold phenomenon provides a powerful theoretical explanation for emergence of eigenvalues from the essential spectrum. In dimensions $d=1,2$ the smallest coupling constant thresholds, i.e., appearance of the smallest (or largest) eigenvalues, are fairly well-studied \cite{H:1985_jot,KhLA:2021_jmaa,KS:1980,S:1976_ann.phys} and it is tacitly related to the norm of a corresponding Birman-Schwinger operator. However, the emergence of second, third, etc. eigenvalues is not quite well-understood and the methods for finding the smallest coupling constant thresholds are not applicable in such problems. 
Moreover, it is quite difficult to express the thresholds in an explicit or implicit form. Therefore, it is not very surprising that the expansions of absorbing second, third, etc. eigenvalues in the continuum case have been obtained assuming a priori their existence (see e.g., \cite{H:1985_jot,KS:1980}).

In the current paper, for the  operator $\hat H_{a,b}(\mu)$   we explicitly find all coupling constant thresholds in terms of $a,b$ and $\hat\dispersion$. In particular, we provide a sufficient and necessary conditions for the existence of discrete spectrum (Theorem \ref{teo:existence_eigenv}). Moreover, as in \cite{H:1985_jot,LKL:2012.tmf,S:1976_ann.phys} we obtain convergent expansions for absorbing eigenvalues at the corresponding coupling constant thresholds (Theorems \ref{teo:depend_eigenv_I} and \ref{teo:depend_eigenv_II}). Hence our results not only justify the analogous expansions in continuum case, but also improve them providing the existence of eigenvalues. 

We remark that when $\hat H_0=-\hat \Delta$ in $\ell^2(\Z^2),$ the complete classification of the discrete spectrum of $\hat H_{a,b}(1)$ in terms of $a$ and $b$ has been established recently in \cite{HMK:2020,LKhKh:2021}. Our results generalize these existence results for a more general class of Hopping matrices (see Figure \ref{fig:dynamics} below).

Another problem related to the coupling constant threshold is the characterization of threshold resonances and thresholds eigenvalues of $\hat H_{\hat v}(\mu).$ Such problems are  especially important in the study of the Efimov and super-Efimov effects (see e.g., \cite{ALM:2004_ahp,BT:2017_ahp,G:2014,NE:2017_rep,NMS:2013,S:1993,T:2019,Y:1974} and the references therein). 

In the current paper, as we consider the discrete spectrum above the essential spectrum, we are interested in high energy resonances. Among various definitions, we mainly follow to \cite{AGK:1982_ahpnc,ALMM:2006_cmp,S:1993,Y:1979}: in the momentum representation, a resonance of energy $\emax,$ where $\emax$ is the top of the essential spectrum, is a nonzero solution $f$ of the eigenvalue equation $H_{a,b}(\mu) f= \emax f$ belonging to $L^1(\T^2)\setminus L^2(\T^2)$;  see also \cite{CS:2001_cmp,E:2013_pro,HMK:2020,L:1992_tmf} and the references therein for other notions of resonances. Using the momentum representation, in Theorem \ref{teo:thresholds} below we completely characterize the threshold eigenfunctions and threshold resonances, also finding them explicitly. It is worth to remark here that the threshold resonances and threshold eigenfunctions are the pointwise limits of eigenfunctions of $\hat H_{a,b}(\mu)$ associated to the absorbing eigenvalues as $\mu$ approaches to that threshold (Remark \ref{rem:origins}).

To prove the main results, we first identify the invariant subspaces of $\hat H_{a,b}(\mu).$ In the momentum representation, under Hypothesis \ref{hyp:main} below it turns out that the Hilbert subspaces of odd-symmetric, odd-antisymmetric, even-antisymmetric and even-symmetric functions in $L^2(\T^2)$ (see Section \ref{sec:main_results} for precise definitions) are  invariant w.r.t. $H_{a,b}(\mu),$ and so, we study the spectrum of $ H_{a,b}(\mu)$ reduced to these subspaces. As in \cite{LKhKh:2021,HMK:2020} we use the Fredholm determinants theory \cite{S:1976_ann.phys,S:1977_adv.math}  to find an implicit equation for the eigenvalues of $H_{a,b}(\mu)$. These equations will be the main tool in identifying the coupling constant thresholds and also in obtaining the convergent expansions of the eigenvalues. We remark that even though the techniques of finding expansions are similar to ones in the continuum \cite{H:1985_jot,KS:1980}, we have to work harder to select non-zero leading terms of the expansion of the determinants especially in the even symmetric subspace (see the proof of Proposition \ref{prop:zeros_delta_es}).

Finite rank property of $\hat V_{a,b}$ and Fredholm's determinant theory also allow to find the explicit form the eigenfunctions including also thresholds. In view of our analysis, odd-symmetric and  odd-antisymmetric eigenfunctions emerge from the threshold resonances, while even-antisymmetric eigenfunctions originates from threshold eigenfunctions. However, the case of even-symmetric eigenfunctions is quite subtle: here the \emph{effect of the Hopping matrice $\hat\dispersion$} comes into play. More precisely,  we can find two numbers $\Theta^*$ and $\Theta^{**}$ explicitly depending only on $\hat \dispersion$ such that if $\Theta^*a\ne\Theta^{**}b,$ then the eigenfunctions of $H_{a,b}(\mu)$ appear  from ``nothing'' (as the eigenfunctions  corresponding to the largest eigenvalue of $H_{a,b}(\mu)$), whereas if $\Theta^*a=\Theta^{**}b,$ then the eigenfunctions emerge from the threshold eigenfunctions (see Theorem \ref{teo:thresholds} and Remark \ref{rem:origins}). We recall that such an ``Hopping matrix effect'' does not exists in the case of discrete laplacian $-\hat \Delta$ (see Section \ref{subsec:discrete_laplace}) and this is the another motivation towards to study a more general classes of $\hat\dispersion.$

When the Hopping $\hat \dispersion(\cdot)$ is even in each coordinates (for instance in the case of discrete Laplacian), then the coupling constant thresholds related to the subspaces of odd-symmetric and odd-antisymmetric functions coincide. Then for any $b>0$ we can choose $a>0$ such that the essential spectrum ``gives birth'' simultaneously three eigenvalues  (Remark \ref{rem:thre_resonaaa}).

The structure of the paper is as follows. In Section \ref{sec:disc_shr_op} we introduce the operator $\hat H_{a,b}(\mu)$ in the coordinate and momentum spaces, and reduce some examples of Hopping matrices $\hat \dispersion$ satisfying Hypothesis \ref{hyp:main}. The main results of the paper -- Theorems \ref{teo:existence_eigenv}-\ref{teo:thresholds} are formulated in Section \ref{sec:main_results}. The main results will be proven in Section \ref{sec:proofs} and in Section \ref{sec:examples} we consider some examples. Finally, we conclude the paper with an appendix containing results on the expansion of some integrals.

\subsection*{Acknowledgment}

Sh.Kh. acknowledges support from the Austrian Science Fund (FWF) project P~33716. S.L. acknowledges support from the Foundation for Basic Research of the Republic of Uzbe\-kistan (Grant No. OT-F4-66).

\section{Discrete Schr\"odinger operator}\label{sec:disc_shr_op}

Let  $\Z^2$ be the two dimensional cubical lattice and
$\T^2:=(-\pi,\pi]^2$ be the two dimensional torus, the dual group to $\Z^2.$ We equip $\T^2$ with the Haar measure.
Let  $\ell^2(\Z^2)$ resp. $L^2(\T^2)$ be the Hilbert space  of square-summable resp. square-integrable functions defined on $\Z^2$  resp. $\T^2.$

In the coordinate space representation the energy operator $\hat H_{a,b}(\mu)$ of a one-particle system on the two-dimensional lattice $\Z^2$ with a potential field
\begin{equation*}
\hat v(x)=
\begin{cases}
a, & \text{if $x=0,$}\\
b, & \text{if $|x|=1,$}\\
0, & \text{if $|x|>1$}
\end{cases}
\end{equation*}
with $a,b\in \R\setminus\{0\},$ is defined as
\begin{equation*}
\hat H_{a,b}(\mu): =\hat H_0+\mu\widehat V_{a,b}, \qquad \mu\geq0,  \end{equation*}
where the free energy operator $\hat H_0$ is a Laurent-Toeplitz-type operator
\begin{equation*} \label{dispersion1}
\hat H_0 \hat f (x) = \sum\limits_{y\in\Z^2} \hat \dispersion(x-y)\hat f(y)
\end{equation*}
in $\ell^2(\Z^2),$ given by a \emph{Hopping matrix} $\hat\dispersion \in \ell^1(\Z^2)$ which satisfies $\hat \dispersion(x) = \overline{\hat \dispersion(-x)},$ and the potential energy operator is the multiplication in $\ell^2(\Z^2)$ by the function $\hat{v}.$  Note that $\hat{H}_\mu$ is a bounded self-adjoint operator.

In the momentum space representation, the operator acts in $L^2(\T)$ by
\begin{equation}\label{def:H_mu999}
H_{a,b}(\mu) = H_{0}+\mu V_{a,b} 
\end{equation}
where
$$
H_0:=\cF\hat H_0\cF^*,\quad V_{a,b}:={\mathcal{F}}\widehat{V}_{a,b}{\mathcal{F}}^*
$$
and
\begin{equation*}
\cF:\ell^2(\Z^2)\rightarrow L^2(\T^2),\quad
\cF\hat{f}(p)=\frac{1}{2\pi} \sum_{x\in\Z^2} \hat{f}(x) e^{\mathrm{i}(p,x)}
\end{equation*}
is the standard Fourier transform  with the inverse
\begin{equation*}
\cF^*:L^2(\T^2)\rightarrow \ell^2(\Z^2),\quad
\cF^*f(x)=\frac{1}{2\pi} \int_{\T^2} f(p) e^{-\mathrm{i}(p,x)}\d p.
\end{equation*}
The free Hamiltonian $H_0$ is the multiplication operator in $L^2(\T^2)$ by the function
\begin{equation*}
\dispersion:=2\pi \cF\hat\dispersion
\end{equation*}
so-called the \emph{dispersion relation} of the particle and the potential $V_{a,b}$ acts on $L^2(\T^2)$ as a convolution operator
\begin{equation*}
V_{a,b}f(p)= \frac{1}{2\pi} \int_{\T^2} v(p-q)f(q)dq
\end{equation*}
with the analytic on $\T^2$ kernel
\begin{align*}
v(p)=\cF\hat v(p)=&\frac{1}{2\pi} \sum_{x\in \Z^2} \hat v(x) e^{ip\cdot x} =\frac{1}{2\pi} \Big(a+2 b \sum_{i=1}^2\cos p_i\Big).
\end{align*}
Since $V_{a,b}$ is compact, by Weyl's Theorem \cite[Theorem XIII.14]{RS:vol.IV},
\begin{equation*}
\sigma_\ess(H_{a,b}(\mu)) =\sigma(H_0)=
[\dispersion_{\min},\dispersion_{\max}],\qquad \mu\ge0,
\end{equation*}
where
$$
\emin:= \min\,\dispersion,\quad \emax:=\max\,\dispersion.
$$

In what follows we always assume:

\begin{hypothesis}\label{hyp:main}
The dispersion relation $\dispersion$ is a real-valued even
function, symmetric with respect to coordinate permutations and having a non-degenerate unique maximum at $\vec{\pi}=(\pi,\pi)\in\T^2.$ Moreover, $\dispersion$ is analytic near $\vec\pi$.
\end{hypothesis}

\begin{remark}\label{rem:morse_lemma}
In view of Hypothesis \ref{hyp:main}
$$
\nabla \dispersion(\vec\pi) = \Big(\frac{\partial \dispersion}{\partial q_1}(\vec{\pi}),\frac{\partial \dispersion}{\partial q_2}(\vec{\pi})\Big)=(0,0)
$$
and the Hessian
$$
\nabla^2\dispersion(\vec\pi)
=
\begin{pmatrix}
\frac{\partial^2 \dispersion}{\partial q_1^2}(\vec{\pi}) & \frac{\partial^2 \dispersion}{\partial q_1\partial q_2}(\vec{\pi})  \\[2mm]
\frac{\partial^2 \dispersion}{\partial q_1\partial q_2}(\vec{\pi}) & \frac{\partial^2 \dispersion}{\partial q_2^2}(\vec{\pi})
\end{pmatrix}
$$
is strictly negative definite. By the symmetricity of $\dispersion$ 
$$
\frac{\partial^2 \dispersion}{\partial q_1^2}(\vec{\pi})=\frac{\partial^2 \dispersion}{\partial q_2^2}(\vec{\pi})<0\quad\text{and} \quad 
\Big|\frac{\p^2}{\p q_1^2}(\vec\pi)\Big| >
\Big|\frac{\p^2}{\p q_1\p q_2}(\vec\pi)\Big|.
$$
Moreover, by the Morse Lemma there exist a neighborhood  $U(\vec{\pi})\subset \T^2$  of  $\vec{\pi}\in\T^2$
and an analytic diffeomorphism $\psi:B_\gamma(0)\to U(\vec{\pi}),$
where $B_\gamma(0)\subset\R^2$ is
a ball in $\R^2$ of radius $\gamma\in(0,1)$ centered at the origin, such that   $\psi(0) = \vec{\pi}$ and
$$
 \dispersion(\psi(y)) = \dispersion_{\max} - y^2,\quad y\in B_\gamma(0).
$$
Moreover, the Jacobian   $J(\psi(y))$ of  $\psi$ are strictly positive in $B_\gamma(0).$  We write
\begin{equation}\label{jacobian}
J_0:=  \frac{1}{4\pi}\,J(\psi(0))>0.
\end{equation}
\end{remark}

Let us consider some examples.

\begin{example}\label{ex:discret_laplas}
Let
$$
\dispersion(p) = 2-\cos p_1 - \cos p_2
$$
so that $\hat H_0$ coincide with the discrete Laplacian
$$
\hat \Delta f(x) = - \frac12\sum\limits_{|s|=1} (f(x+s) - f(x))
$$
in $\ell^2(\Z^2).$
In this case $\emax=\dispersion (\vec{\pi})=4,$
$
\nabla \dispersion(\vec{\pi})=(0,0)
$
and
$
-\nabla^2\dispersion(\vec\pi)
$
is the identity. Moreover, the analytic function
$$
\psi(y_1,y_2)=\vec{\pi} - \Big(2\arcsin \frac{y_1}{\sqrt{2}},2\arcsin \frac{y_2}{\sqrt{2}}\Big)
$$
satisfies
$$
\dispersion(\psi(y))=4-y_1^2-y_2^2=4-y^2.
$$
The Jacobian of $\psi$ satisfies
$$
J(\psi(y))=\frac{4}{\sqrt{2-y_1^2}\sqrt{2-y_2^2}}>0\quad \text{and}\quad J(\psi(0))=2
$$
so that $J_0=\frac{1}{2\pi}.$
\end{example}

\begin{example} \label{ex:almost_laplas_sumo}
Let $\hat \dispersion\in\ell^1(\Z^2)$ be an even symmetric real-valued function such that

\begin{itemize}
 \item[(a1)] $\hat \dispersion(x)\le0$ resp. $\hat\dispersion(x)=0$ for all $x=(x_1,x_2)\in\Z^2$  with $x_1x_2$ is odd resp. even;

 \item[(a2)] $\hat \dispersion(x)<0$ if  $|x|=1;$

 \item[(a3)] there exists $C,\alpha>0$ such that $|\hat\dispersion(x)| \le Ce^{-\alpha|x|}$  for all $x\in\Z^2.$
\end{itemize}
Then $\dispersion$ is real-analytic and $\vec\pi$ is its non-degenerate unique maximum point
(see also \cite[Example 1.2]{KhLA:2021_jmaa}).
\end{example}

\begin{example}\label{ex:stran_epsilon}
Given $\epsilon\in(0,1),$ let
$$
\phi(t)=
\begin{cases}
-\cos t - \cos \epsilon &\text{if $t\in(-\pi,-\pi+\epsilon)\cup (\pi-\epsilon,\pi],$}\\
0 & \text{if $[-\pi+\epsilon,\pi-\epsilon]$}
\end{cases}
$$
and set
$$
\dispersion(p) = \phi(p_1) +\phi(p_2).
$$
Obviously, $\dispersion(\cdot)$ is a Lipschitz function, real-analytic near its non-degenerate unique maximum point $\vec\pi.$
Since the one-dimensional Fourier coefficients are defined as
$$
\hat\phi(n) =
\begin{cases}
2\sin\epsilon - 2\epsilon\cos\epsilon &\text{if $n=0,$}\\
- \epsilon & \text{if $|n|=1,$}\\
\frac{(-1)^n\sin(|n|-1)\epsilon}{|n|(|n|-1)} - \frac{(-1)^n\sin(|n|+1)\epsilon}{|n|(|n|+1)}& \text{if $|n|\ne 0,1,$}
\end{cases}
$$
$\hat\phi\in \ell^1(\Z)$ and hence, by the linearity and boundedness of the Fourier transform $
\hat\dispersion \in \ell^1(\Z^2).$
\end{example}

\section{Main results}\label{sec:main_results}

Let $L^{2,\mathrm{e}}(\T^2)$ and $L^{2,\mathrm{o}}(\T^2)$ be the subspaces of essentially even and essentially odd functions in $L^2(\T^2)$ and let
\begin{equation*}
\begin{aligned}
L^{2,\es}(\T^2):= & \{f\in L^{2,\mathrm{e}}(\T^2): f(p_1,p_2)=f(p_2,p_1) \,\,\text{for a.e}\,\,(p_1,p_2) \in \T^2\},\\
L^{2,\ea}(\T^2):= & \{f\in L^{2,\mathrm{e}}(\T^2): f(p_1,p_2)=-f(p_2,p_1) \,\,\text{for a.e}\,\,(p_1,p_2) \in \T^2\},\\
L^{2,\os}(\T^2):= & \{f\in L^{2,\mathrm{o}}(\T^2): f(p_1,p_2)=f(p_2,p_1) \,\,\text{for a.e}\,\,(p_1,p_2) \in \T^2\},\\
L^{2,\oa}(\T^2):= & \{f\in L^{2,\mathrm{o}}(\T^2): f(p_1,p_2)=-f(p_2,p_1)\,\,\text{for a.e}\,\, (p_1,p_2) \in \T^2\}
\end{aligned}
\end{equation*}
be the (Hilbert) subspaces of (essentially) even-symmetric, even-antisymmetric, odd-symmetric and odd-antisymmetric functions in $L^2(\T^2),$ respectively. Recall that
$$
L^2(\T^2) = L^{2,\os}(\T^2)\oplus L^{2,\oa}(\T^2)\oplus L^{2,\ea}(\T^2)\oplus L^{2,\es}(\T^2).
$$
Since $\dispersion$ is even and symmetric, for each $\omega\in\{\os,\oa,\ea,\es\}$ the subspace $L^{2,\omega}(\T^2)$  is invariant w.r.t. $H_0.$ Recalling that
\begin{align*}
V_{a,b}f(p) = \frac{a}{2\pi}\int_{\T^2}f(q)\d q & + \sum\limits_{i=1}^2
\frac{b}{2\pi}\int_{\T^2} (\cos p_1\cos q_1 + \cos p_2\cos q_2)f(q)\d q \\
& + \sum\limits_{i=1}^2 \frac{b}{2\pi}
\int_{\T^2} (\sin p_1\sin q_1 + \sin p_2\sin q_2)f(q)\d q
\end{align*}
and writing
\begin{align*}
2\cos p_1\cos q_1 + & 2\cos p_2\cos q_2 \\
=  &(\cos p_1 + \cos p_2)(\cos q_1 + \cos q_2) + (\cos p_1 - \cos p_2)(\cos q_1 - \cos q_2),\\
2\sin  p_1\sin  q_1 + & 2\sin  p_2\sin  q_2 \\
= & (\sin  p_1 + \sin  p_2)(\sin  q_1 + \sin  q_2) +  (\sin  p_1 - \sin  p_2)(\sin  q_1 - \sin  q_2)
\end{align*}
we get that each $L^{2,\omega}$ is invariant also w.r.t. $V_{a,b}.$ Therefore,
\begin{equation}\label{sigma_Hmu}
\sigma(H_{a,b}(\mu))= \bigcup_{\omega\in\{\os,\oa,\ea,\es\}} \sigma\big(H_{a,b}(\mu)\big|_{L^{2,\omega}(\T^2)}\big),
\end{equation}
where $A\big|_\cH$ is the restriction of a self-adjoint operator on the subspace $\cH.$
Thus, we study the discrete spectrum of $H_{a,b}(\mu)$ separately restricted to these subspaces.

\subsection{Existence of eigenvalues}
Let
\begin{subequations}
\begin{align}
\gamma_{\os}:= & \Big(\frac{1}{4\pi^2}\int_{\T^2}\frac{(\sin q_1 +\sin q_2)^2\, \d q}{\emax-\dispersion(q)}\Big)^{-1},\label{def:mu_os}\\
\gamma_{\oa}:= & \Big(\frac1{4\pi^2}\int_{\T^2}\frac{(\sin q_1 -\sin q_2)^2\, \d q}{\emax-\dispersion(q)}\Big)^{-1},
\label{def:mu_oa} \\
\gamma_{\ea}:= & \Big(\frac1{4\pi^2}\int_{\T^2}\frac{(\cos q_1 -\cos  q_2)^2\, \d q}{\emax-\dispersion(q)}\Big)^{-1}, \label{def:mu_ea} 
\\
\gamma_{\es}:= & \Big(\frac{1}{4\pi^2}\int_{\T^2} \frac{(\cos q_1 + \cos q_2+2)^2 \d q}{\emax - \dispersion(q)}\Big)^{-1}. \label{def:mu_es}
\end{align}
\end{subequations}
By Corollary \ref{cor:integral_finite00}  all integrals in \eqref{def:mu_os}-\eqref{def:mu_es} are finite and positive.
These numbers play an important role in the definition of coupling constant thresholds.

\begin{theorem}
\label{teo:existence_eigenv}
Given $a,b\in\R\setminus \{0\},$ let
$$
\mu_\omega^0:=
\begin{cases}
\tfrac{\gamma_\omega}{b} & \text{if $\omega \in\{\os,\oa,\ea\},$} \\
\tfrac{a+4b}{ab}\,\gamma_\omega & \text{if $\omega=\es$ and $\frac{a+4b}{ab}>0,$}\\
0 & \text{if $\omega=\es$ and $\frac{a+4b}{ab}\le 0,$}
\end{cases}
$$
where the positive numbers $\gamma_\omega$ are given by \eqref{def:mu_os}-\eqref{def:mu_es}.
\smallskip

(a) Let $\omega\in\{\os,\oa,\ea\}.$

\begin{itemize}
 \item Assume that either $b<0$ and $\mu>0$ or $b>0$ and $\mu\in(0,\mu_\omega^0].$
  Then $H_{a,b}(\mu)\big|_{L^{2,\omega}(\T^2)}$ has no eigenvalues above the essential spectrum.

 \item Assume that $b>0.$ Then for  $\mu>\mu_\omega^0$ then operator $H_{a,b}(\mu)\big|_{L^{2,\omega}(\T^2)}$ has a unique eigenvalue $E_\omega(\mu)>\emax$ and the associated eigenfunction is
 $$
  \Psi_\mu^{\omega}(p) =
  \begin{cases}
  \frac{\sin p_1 + \sin p_2}{E_\omega(\mu) - \dispersion(p)} & \text{if  $\omega=\os$},\\ 
  \frac{\sin p_1 - \sin p_2}{E_\omega(\mu) - \dispersion(p)} & \text{if $\omega=\oa,$ }\\ 
  \frac{\cos p_1 - \cos p_2}{E_\omega(\mu) - \dispersion(p)} & \text{if $\omega=\ea.$ }
  \end{cases}
  $$
 Moreover, $E_\omega(\cdot)$ is real-analytic, strictly increasing and strictly convex in $(\mu_\omega^0,+\infty).$ 
\end{itemize}

(b) Let $\omega=\es.$

\begin{itemize}
 \item Assume that either $a,b<0$ and $\mu>0$ or $ab<0$ and $\mu\in(0,\mu_{\es}^0].$ Then $H_{a,b}(\mu)\big|_{L^{2,\es}(\T^2)}$ has no eigenvalues above the essential spectrum.

 \item Assume that $ab<0$ Then for any $\mu>\mu_{\es}^0$ the  operator $H_{a,b}(\mu)\big|_{L^{2,es}(\T^2)}$ has a unique eigenvalue $E_{\es}(\mu)>\emax$ and the associated eigenfunction is
 $$
 \Psi_\mu^{\es} (p) = \frac{C_\mu^1 + C_\mu^2(\cos p_1 + \cos p_2)}{E_{\es}(\mu) - \dispersion(p)}
 $$
 for real some constants $C_\mu^1,C_\mu^2$ (see Remark \ref{rem:eigenvectorsss}). Moreover, $E_{\es}(\cdot)$ is is real-analytic, strictly increasing and  convex in $(\mu_\omega^0,+\infty).$

 \item Assume that $a,b>0.$ Then for any $\mu\in(0,\mu_{\es}^0]$ the operator $H_{a,b}(\mu)\big|_{L^{2,\es}(\T^2)}$ has a unique eigenvalue $E_{\es}^{(1)}(\mu)>\emax$ and for any $\mu>\mu_{\es}^0$ $H_{a,b}(\mu)\big|_{L^{2,\es}(\T^2)}$  has two eigenvalues $E_{\es}^{(1)}(\mu),E_{\es}^{(2)}(\mu)>\emax$, and the associated eigenfunctions are
 $$
 \Psi_\mu^{\es,i} (p) = \frac{C_\mu^{1,i} + C_\mu^{2,i}(\cos p_1 + \cos p_2)}{E_{\es}^i(\mu) - \dispersion(p)},\quad i=1,2,
 $$
 for some real constants $C_\mu^{1,i},C_\mu^{2,i}$ (see Remark \ref{rem:eigenvectorsss}). Moreover, $E_{\es}^{(1)}(\cdot)$ and $E_\es^{(2)}(\cdot)$ are real-analytic and strictly increasing in $(0,+\infty)$ and $(\mu_\omega^0,+\infty),$ respectively.
\end{itemize}
\end{theorem}

For fixed $\mu>0,$ a schematic location of eigenvalues of $H_{a,b}(\mu)$ depending on $a$ and $b$ is drawn in Figure \ref{fig:dynamics}.

\begin{figure}[h!]
\includegraphics[width=0.6\textwidth]{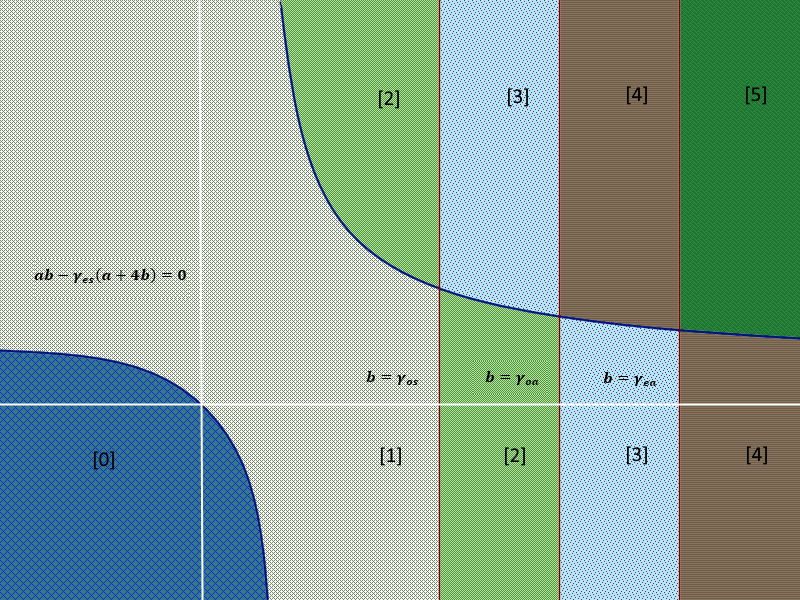}
\caption{A schematic dynamics of the change in the number of eigenvalues depending on parameters $a$ and $b$ for $\mu=1:$ in the region labelled with [n] the operator $H_{a,b}(1)$ has $n$ eigenvalues above the essential spectrum.}\label{fig:dynamics}
\end{figure}

\begin{remark}
In general we cannot state $E_{\es}^{(1)}(\mu)>E_{\es}^{(2)}(\mu)$ for all $\mu>\mu_{\es}^0$ (see Section \ref{subsec:multiple_eigen}). However, in the case of discrete Laplacian this inequality always holds (see Section \ref{subsec:discrete_laplace}).
\end{remark}

\subsection{Asymptotics of eigenvalues of {$H_{a,b}(\mu)$ }}
In this section we establish the absoption rate of eigenvalues of $H_{a,b}(\mu)$ by the essential spectrum.
First consider the case $\omega\in\{\os,\oa,\ea\}.$

\begin{theorem}\label{teo:depend_eigenv_I}
Let $\omega\in\{\os,\oa,\ea\}$ and $b>0.$  Let $E_\omega(\mu)>\emax$ be the unique eigenvalue of $H_{a,b}(\mu)\big|_{L^{2,\omega}(\T^2)}.$  Then for sufficiently small and positive $\mu-\mu_\omega^0$
the function $E_\omega(\mu) - \emax$ has a convergent expansion
$$
E_\omega(\mu) - \emax =
 \begin{cases}
  c_{\omega}\tau + \sum\limits_{n,m,k\geq0}c_{nmk}^\omega \lambda^{n}\tau^{m+1}\theta^k  &
 \text{if $\omega = \os,\oa,$}\\[3mm]
 c_{\omega}\lambda + \sum\limits_{n,m\ge0} c_{nm}^{\omega}\lambda^{n+1}\sigma^m  &  \text{if $\omega=\ea,$}
 \end{cases}
$$
where$\{c_{nmk}^{\os}\},$ $\{c_{nmk}^{\oa}\}$ and $\{c_{nm}^{\ea}\}$ are real coefficients,
$$
 \lambda:=\mu-\mu_\omega^0,\quad \tau:=\tfrac{\lambda}{-\ln\lambda},\quad
 \theta:=\tfrac{\ln\ln\lambda^{-1}}{-\ln\lambda},\quad \sigma:=-\lambda\ln\lambda,
$$
and
$$
c_\omega =
 \begin{cases}
 \frac{2}{b J_0\mu_\omega^2}\,\Big[ \Big(\frac{\p \psi}{\p y_1}(\vec{0})\Big)^2+\Big(\frac{\p \psi}{\p y_2}(\vec{0})\Big)^2\Big]^{-1} & \text{if $\omega=\oa,\os,$}\\[2mm]
 \frac{1}{b\,\mu_{\omega}^2}\Big(\frac{1}{4\pi^2}\int_{\T^2}\frac{(\cos q_1 -\cos q_2)^2\, \d q}{(\emax-\dispersion(q))^2}\Big)^{-1}
 & \text{if $\omega=\ea$}
 \end{cases}
$$
is a finite positive real number, here $\psi$ and $J_0$ are as in Remark \ref{rem:morse_lemma}.
\end{theorem}

Next we study the case $\omega=\es.$ Let us introduce first the following real numbers:
\begin{subequations}
\begin{align}
& \Theta^*:=  \frac{1}{4\pi^2}\int_{\T^2} \frac{(2+\cos p_1 + \cos p_2)\d p}{\emax- \dispersion(p)},\label{theta_mayor}\\
& \Theta^{**}:=  \frac{1}{4\pi^2} \int_{\T^2} \frac{2(\cos p_1+\cos p_2)(2+\cos p_1 + \cos p_2)\d p}{\emax- \dispersion(p)},\label{theta_podpolkovnik}\\
& \kappa_1:= \frac{1}{4\pi^2} \int_{\T^2} \frac{[4-(\cos p_1 +\cos p_2)^2]\d p}{\emax - \dispersion(p)}, \label{kappa1}
\end{align}
\end{subequations}
Note that by Corollary \ref{cor:integral_finite00} all these numbers are finite and by the nonnegativity of corresponding integrands, $\Theta^*$ and $\kappa_1$ are positive.

\begin{theorem}\label{teo:depend_eigenv_II}
Let $E_{\es}(\cdot),E_{\es}^{(1)}(\cdot)$ and $E_{\es}^{(2)}(\cdot)$ be  as in Theorem \ref{teo:existence_eigenv} (b).
\smallskip

\noindent
(a) Assume that either $\frac{a+4b}{ab}\le0$ or $a,b>0$ and let $X_{a,b}:(0,+\infty)\to \R$ be defined as
\begin{equation}\label{def:Xab09}
X_{a,b}(\mu)=
\begin{cases}
E_{\es}(\mu)  & \text{if $\frac{a + 4b}{ab}\le0$,}\\
E_{\es}^{(1)}(\mu) & \text{if $a,b>0.$}
\end{cases}
\end{equation}
Then for sufficiently small and positive $\mu$
the function $X_{a,b}(\mu) - \emax$ has a convergent expansion
\begin{align}\label{esX000}
X_{a,b}(\mu) - \emax
=
 \begin{cases}
 c_{\es} \, e^{-\frac{1}{J_0(a+4b) \mu}}
+\sum\limits_{\substack{n,m\geq1,\\ n+m\ge3}} c_{nm}^{(1)}\, \mu^n\Big(\frac{1}\mu e^{-\frac{1}{J_0(a+4b) \mu}}\Big)^m  & \text{if $a+4b>0,$}\\[3mm]
c_{\es}\,e^{\frac{1+ b{ \kappa_1} \mu}{J_0\gamma_\es^{-1} ab  \mu^2}}+\sum\limits_{\substack{n,m\geq1,\\ n+m\ge 3}}c_{nm}^{(2)}\, \mu^{n+2}\Big(\frac{1}{\mu^3}e^{\frac{1+ b\kappa_1 \mu}{J_0\gamma_\es^{-1} ab \mu^2}}\Big)^m & \text{if $a+4b=0,$}
\end{cases}
\end{align}
where $J_0$ is given by \eqref{jacobian}, $\{c_{nm}^{(1)}\}$ and $\{c_{nm}^{(2)}\}$  are real coefficients,
$$
 c_{\es} =
 \begin{cases}
 e^{\frac{\beta_1a^2 + \beta_2ab  +\beta_3 b^2}{(a+4b)^2}} & \text{if $a+4b>0,$}\\
 e^{\beta_0} & \text{if $a+4b=0,$}
 \end{cases}
$$
and $\beta_i $ are some universal coefficients depending only on $\dispersion.$
\smallskip

\noindent
(b) Assume that $\frac{a+4b}{ab}>0$ and let $Y_{a,b}:(\mu_\es^0,+\infty)\to \R$ be defined as
\begin{equation}\label{def:Yab09}
Y_{a,b}(\mu)=
\begin{cases}
E_{\es}(\mu)  & \text{if $ab<0$ and $a+4b<0$,}\\
E_{\es}^{(2)}(\mu) & \text{if $a,b>0.$}
\end{cases}
\end{equation}
Then for sufficiently small and positive $\mu-\mu_{\es}^0$
the function $Y_{a,b}(\mu) - \emax$ has the following convergent expansions:
\begin{itemize}
\item if $\Theta^*a\ne \Theta^{**}b,$ then
\begin{equation}\label{esY000}
Y_{a,b}(\mu) - \emax =  c_{\es} \, e^{-\frac{\Lambda}{\mu-\mu_{\es}^0}} + \sum\limits_{\substack{n,m\geq1,\\ n+m\ge3}} c_{nm}^{(3)}\, (\mu-\mu_\es^0)^n\Big(\tfrac{1}{\mu-\mu_{\es}^0} e^{-\frac{\Lambda}{\mu-\mu_{\es}^0}}\Big)^m,
\end{equation}
where
$$
\Lambda:=\frac{\gamma_\es^2(\Theta^*a-\Theta^{**}b)^2}{J_0ab(a+4b)}>0,\quad c_{\es} = e^{\beta_0' + \frac{\beta_2' a + \beta_3'b}{a+4b} + \frac{\beta_4'a^2 + \beta_5'ab  +\beta_6' b^2}{(a+4b)^2}}>0,
$$
$\{c_{nm}^{(3)}\}$  are real coefficients 
and $\{\beta_i'\}$ are coefficients depending only on $\dispersion;$

 \item if $\Theta^*a=\Theta^{**}b,$ then
 sufficiently small and positive $\mu - \mu_{\es}^0>0$  the function $Y_{a,b}(\mu) - \emax$ has a convergent expansion
\begin{equation}\label{esY001}
 Y_{a,b}(\mu) - \emax =
 c_{\es} \lambda + \sum\limits_{n,m,k\ge0} c_{nmk}\lambda^{n+1}\eta^m\sigma^k,
\end{equation}
 where
 $$
 \lambda: = \mu-\mu_{\es},\quad \eta:=\frac{1}{-\ln\lambda},\quad \sigma:=-\lambda\ln\lambda,
 $$
 the coefficients $\{c_{nmk}\}$ are real and
 $$
 c_{\es}:= \frac{ab}{(a+4b)\gamma_\es^2}\Big(\frac{1}{4\pi^2}\int_{\T^2} \frac{(2 + \cos p_1+\cos p_2)^2\d p}{(\emax-\dispersion(p))^2}\Big)^{-1}>0.
 $$
\end{itemize}
\end{theorem}

\subsection{Threshold eigenfunctions and threshold resonances}

Consider the equation 
\begin{equation}\label{eigen_equation}
H_{a,b}(\mu_\omega^0) f= \emax f
\end{equation}
for $\omega\in\{\os,\oa,\ea,\es\}$ and 
for some nonzero $f\in L^{2,\omega}(\T^d)$ provided $\mu_{\omega}>0.$ Since the domain of the operator $H_{a,b}(\mu)$  can be extended to $L^1(\T^2)$ preserving boundedness, we look for a (generalized) solution of  \eqref{eigen_equation} in $L^{1,\omega}(\T^2).$
Recall that \cite{ALM:2004_ahp} a solution $f\in L^{1,\omega}(\T^2)$ of \eqref{eigen_equation} is called \emph{threshold resonance} if $f\notin L^{2,\omega}(\T^2).$ Any solution $f\in L^{2,\omega}(\T^2)$ of \eqref{eigen_equation} is called \emph{threshold eigenfunction}. Next we study threshold resonances and threshold eigenfunctions of $H_{a,b}(\mu)$.

\begin{theorem}\label{teo:thresholds}
(a) Assume that $b>0$ and $\omega\in \{\os,\oa\}.$ Then the function
$$
 \Phi_{\omega}(p) =
  \begin{cases}
  \frac{\sin p_1 + \sin p_2}{\emax - \dispersion(p)} & \text{if  $\omega=\os$},\\
  \frac{\sin p_1 - \sin p_2}{\emax - \dispersion(p)} & \text{if $\omega=\oa$ }
  \end{cases}
$$
belongs to $L^{1,\omega}(\T^2)\setminus L^{2,\omega}(\T^2)$ and solves \eqref{eigen_equation}.
\medskip

(b) Assume that $b>0$ and $\omega=\ea.$ Then the function
$$
\Phi_{\ea}(p)=  \dfrac{\cos p_1 - \cos p_2}{\emax - \dispersion(p)}
$$
belongs to $L^{2,\omega}(\T^2)$ and solves \eqref{eigen_equation}.
\medskip

(c) Assume that $\omega=\es,$ $\frac{a+4b}{ab}>0$ and  $\Theta^*a=\Theta^{**}b,$ where $\Theta^*$ and $\Theta^{**}$ are given by \eqref{theta_mayor} and \eqref{theta_podpolkovnik}. Then the function
$$
\Phi_{\es}(p)=  \dfrac{2+\cos p_1 + \cos p_2}{\emax - \dispersion(p)}
$$
belongs to $L^{2,\omega}(\T^2)$ and solves \eqref{eigen_equation}.
\medskip

(d) Assume that $\omega=\es,$ $\frac{a+4b}{ab}>0$ and  $\Theta^*a\ne \Theta^{**}b.$ Then  \eqref{eigen_equation} has only $0$-solution in $L^{1,\es}(\T^2)$.
\end{theorem}

\begin{remark}\label{rem:origins}
The functions  $\Phi_\omega$ are nothing but the pointwise limits of the eigenfunctions $\Psi_\omega^\mu(p)$ of $H_{a,b}(\mu)$ (given by Theorem \ref{teo:existence_eigenv}) as $\mu\searrow\mu_\omega^0$, and hence, they are  the ``origin''s of eigenfunctions. In this sense, we can say that odd-symmetric and
odd-antisymmetric eigenfunctions come out from threshold resonances and even-antisymmetric functions appear from threshold eigenfunctions.
The case of even-symmetric eigenfunctions are bit delicate: eigenfunctions may appear ``mostly'' from nothing, and in ``rare'' cases, appear from threshold eigenfunctions.
\end{remark}

\begin{remark}\label{rem:thre_resonaaa}
When $\dispersion(\cdot)$ is even in each coordinates, then by \eqref{def:mu_os} and \eqref{def:mu_oa} one has $\gamma_\os=\gamma_\oa.$ In particular, if we choose $a>0$ such that 
$\mu_\os=\mu_\oa=\mu_\es$ for any $b>0,$ then three eigenvalues simultaneously release from the essential spectrum. 
\end{remark}

\section{Proofs of main results}\label{sec:proofs}

In what follows we write
\begin{align*}
V_{b}^{\os}f(p) := & \frac{b}{4\pi^2}\,(\sin p_1 +\sin p_2)\int_{\T^2}(\sin q_1+\sin q_2)f(q)\d q,\\
V_{b}^{\oa}f(p) := & \frac{b}{4\pi^2}\,(\sin p_1 - \sin p_2) \int_{\T^2} (\sin q_1 - \sin q_2)f(q)\,\d q\\
V_{b}^{\ea}f(p) := & \frac{b}{4\pi^2}\,(\cos p_1 - \cos p_2) \int_{\T^2} (\cos q_1 - \cos q_2)f(q)\,\d q,\\
V_{a,b}^{\es}f(p) := & \frac{a}{4\pi^2}\int_{\T^2} f(q)\,\d q + \frac{b}{4\pi^2}\,(\cos p_1 + \cos p_2) \int_{\T^2} (\cos q_1 + \cos q_2)f(q)\,\d q.
\end{align*}
Then
$$
H_{a,b}(\mu)\big|_{L^{2,\omega}(\T^2)} = H_b^\omega(\mu):=H_0 + \mu V_{b}^\omega
$$
for $\omega\in\{\os,\oa,\ea\}$ and 
$$
H_{a,b}(\mu)\big|_{L^{2,\omega}(\T^2)} = H_{a,b}^\es(\mu):=H_0 + \mu V_{a,b}^\es,
$$
and the essential spectrum of $H_{a,b}(\mu)\big|_{L^{2,\omega}(\T^2)}$ coincides with the segment $[\emin,\emax].$ 

\subsection{Eigenvalues of rank-one perturbations of \text{$H_0$}}\label{subsec:eigen_rank1}

In this subsection we study the eigenvalues of operators of the form
$$
H_v(\mu):=H_0 + \mu v(p)\int_{\T^2} v(q)f(q)\d q
$$
defined in $L^2(\T^2),$ where $v(\cdot)$ is a given nonzero real-analytic function on $\T^2$. By the min-max principle, $H_v(\mu)$ can have at most one eigenvalue  outside the essential spectrum. Also note that $z_0>\emax$ is eigenvalue of $H_v(\mu)$ if and only if $z_0$ is a zero of the corresponding Fredholm determinant
\begin{equation}\label{fred_deter_rank1}
\Delta_v(\mu;z):=1 - \mu\int_{\T^2} \frac{v(q)^2\d q}{z- \dispersion(q)}.
\end{equation}
Such an equivalence will be frequently used subsequently.

First we study the existence of eigenvalues of $H_v(\mu).$

\begin{proposition}\label{prop:existence_rank1}
Let $v$ be nonzero real-analytic function on $\T^2$ and let
$$
\mu_v^0:=\Big[\int_{\T^2}\frac{v(q)^2\d q}{\emax- \dispersion(q)}\Big]^{-1}.
$$
Then $\mu_v^0=0$ if $v(\vec\pi)\ne0$ and $\mu_v^0\in(0,+\infty)$ if $v(\vec\pi)=0.$ Moreover, for any $\mu>\mu_v^0$ the operator $H_v(\mu)$ has a unique eigenvalue $z_v(\mu)>\emax$ and the corresponding eigenfunction is
$$
f_\mu(p) = \frac{v(p)}{z_v(\mu) - \dispersion(p)}.
$$
Moreover, the function $z_v(\cdot)$ is real-analytic, strictly increasing and strictly convex in $(\mu_v^0,+\infty)$ with the asymptotics
\begin{equation}\label{small_mu_asympo}
z_v(\mu)\searrow \emax\quad\text{as $\mu\searrow \mu_v^0$.}
\end{equation}
Finally, if $v\in L^{2,\omega}(\T^2)$ for some $\omega\in\{\os,\oa,\ea,\es\},$ then $L^{2,\omega}(\T^2)$ is invariant subspace of $H_v(\mu)$ and $f_\mu\in L^{2,\omega}(\T^2).$
\end{proposition}

We drop the proof since the all assertions but the last one can be done along the lines of \cite[Theorem 2.1]{KhKhP:2020} using Proposition \ref{prop:kernel_asymptotics_max}. The last assertion is obvious.

Next we study the asymptotics of $z_v(\mu)$ as $\mu\searrow\mu_v^0.$

\begin{proposition}\label{prop:asym_rank1}
(a) Assume that $v(\vec\pi)\ne0.$ Then for sufficiently small and positive $\mu>0$
\begin{align}\label{asymp_at_v2pi0}
z_v(\mu) =\emax+ c_v\,e^{-\frac{1}{4\pi^2J_0 v(\vec{\pi})^2\,\mu} } +\sum\limits_{\substack{n,m\geq0,\\n+m\ge1}}c_{nm} \mu^n\,\Big(\tfrac{1}{\mu}\,e^{-\frac{1}{4\pi^2 J_0v(\vec{\pi})^2\,\mu}}\Big)^{m+1},
\end{align}
where $J_0>0$ is given by Remark \ref{rem:morse_lemma}, $\{c_{nm}\}$ are real  coefficients,
$
c_v:=e^{\frac{\omega_\dispersion}{4\pi^2\,J_0v(\vec{\pi})^2}}
$
and $\omega_\dispersion\in\R$ is a constant depending only on $\dispersion;$

(b) Assume that $v(\vec\pi)=0$ and $\nabla v(\vec\pi)\ne0.$ Then for sufficiently small and positive $\mu-\mu_v^0$
\begin{align}\label{asymp_at_vpi0mas}
z_v(\mu)& =\emax+ \frac{c_v\,(\mu-\mu_v^0)}{-\ln(\mu-\mu_v^0)} \nonumber \\
&+\sum\limits_{\substack{n,m,k\geq0,\\n+m+k\ge1}}c_{nmk} (\mu-\mu_v^0)^{n}\Big( \frac{\mu-\mu_v^0}{-\ln(\mu-\mu_v^0)}\Big)^{m+1}
 \Big(\frac{\ln\ln(\mu-\mu_v^0)^{-1}}{-\ln(\mu-\mu_v^0)}\Big)^k,
\end{align}
where $\{c_{nmk}\}$ are real  coefficients and 
$$
c_v= \frac{1}{2\pi^2 J_0(\mu_v^0)^2} \Big[\Big(\frac{\p v}{\p q_1}(\vec{\pi})\Big)^2\,\Big(\frac{\p \psi}{\p y_1}(\vec{0})\Big)^2+\Big(\frac{\p v}{\p q_2}(\vec{\pi})\Big)^2\,\Big(\frac{\p \psi}{\p y_2}(\vec{0})\Big)^2\Big]^{-1} ;
$$

(c) Assume that $v(\vec\pi)=0$ and $\nabla v(\vec\pi)=0.$ Then
for sufficiently small and positive $\mu-\mu_v^0$
\begin{align}\label{asymp_at_vpi0}
z_v(\mu) =\emax+ c_v\,(\mu-\mu_v^0) +\sum\limits_{\substack{n,m\geq0,\\n+m\ge1}}c_{nm} (\mu-\mu_v^0)^{n+1}(-(\mu-\mu_v^0)\ln(\mu-\mu_v^0))^m,
\end{align}
where $\{c_{nm}\}$ are real  coefficients  and
$$
c_v=\frac{1}{(\mu_v^0)^2}\Big[\int_{\T^2}\frac{v(q)^2\, \d q}{(\emax-\dispersion(q))^2}\Big]^{-1}>0.
$$
\end{proposition}

\begin{proof}
By Proposition \ref{prop:existence_rank1} $\mu_v^0=0$ if $v(\vec\pi)=0$ and $\mu_v^0>0$ if $v(\vec\pi)\ne0.$
Recall that $z:=z_v(\mu)>\emax$ is an eigenvalue of $H_v(\mu)$ if and only if its Fredholm determinant $\Delta^v$ given by \eqref{fred_deter_rank1} satisfies
\begin{equation}\label{delta_teng0_rank1}
\Delta_v(\mu;z)=1 - \mu\int_{\T^2}\frac{v(q)^2\d q}{z - \dispersion(q)}=0.
\end{equation}
By \eqref{small_mu_asympo} $z\searrow\emax$ as $\mu\searrow \mu_v^0.$ From Proposition \ref{prop:kernel_asymptotics_max} and Corollary \ref{cor:integral_finite00}  applied with $v:=v^2$  for sufficiently small $\lambda:=\mu-\mu_v^0>0$ (so that $\alpha:=z-\emax>0$ is also small) the equation \eqref{delta_teng0_rank1} is represented as follows:
\medskip

\begin{itemize}
\item[(a)] if $v(\vec\pi)\ne0$,
\begin{equation}\label{eq_v_noteng0}
\frac{1}{\mu} = \omega_0^{(2)}-4\pi^2 J_0v(\vec{\pi})^2\,\ln \alpha+\sum\limits_{n\geq1}\omega_{n}^{(1)}\alpha^n \ln\alpha +\sum\limits_{n\geq1}\omega_{n}^{(2)}\alpha^n;
\end{equation}

\item[(b)] if $v(\vec\pi)=0$ and $\nabla v(\vec\pi)\ne0,$
\begin{equation}\label{eq_v_ten01}
-\frac{\lambda}{\mu_{v}^0\lambda +(\mu_{v}^0)^2}=\sum\limits_{n\geq1}\omega_{n}^{(1)}\alpha^n \ln\alpha + \sum\limits_{n\geq1}\omega_{n}^{(2)}\alpha^n,
\end{equation}
where 
$$
\omega_1^{(1)} = 2\pi^2 J_0 \Big[\Big(\frac{\p v}{\p q_1}(\vec{\pi})\Big)^2\,\Big(\frac{\p \psi}{\p y_1}(\vec{0})\Big)^2+\Big(\frac{\p v}{\p q_2}(\vec{\pi})\Big)^2\,\Big(\frac{\p \psi}{\p y_2}(\vec{0})\Big)^2\Big]  >0;
$$

\item[(c)] if $v(\vec\pi)=0$ and $\nabla v(\vec\pi)=0,$
\begin{equation}\label{eq_v_ten02}
-\frac{\lambda}{\mu_{v}^0\lambda +(\mu_{v}^0)^2}=\sum\limits_{n\geq1}\omega_{n}^{(2)}\alpha^n  + \sum\limits_{n\geq2}\omega_{n}^{(1)}\alpha^n \ln\alpha,
\end{equation}
where
$$
\omega_1^{(2)}=-\int_{\T^2} \frac{v(q)^2\,\d q}{( \emax-\dispersion(q) )^2}< 0.
$$
\end{itemize}
Here the coefficients $\{\omega_n^{(1)}\},\omega_n^{(2)}\}$ in \eqref{eq_v_noteng0}-\eqref{eq_v_ten02} are real numbers. Now using a singular change of variables, we reduce the equations to the implicit function theorem in analytical case.

(a) To find the implicit function $\alpha=\alpha(\mu)$ solving \eqref{eq_v_noteng0} we set
\begin{equation}\label{alpha_d299}
\alpha:=e^{-\frac{1}{4\pi^2 J_0v(\vec{\pi})^2\,\mu}}(u+d_0),\quad \tau:=\tfrac{1}{\mu}e^{-\frac{1}{4\pi^2 J_0v(\vec{\pi})^2\mu}},
\end{equation}
where
$
d_0:=e^{\frac{\omega_0^{(2)}}{4\pi^2 J_0v(\vec{\pi})^2}}>0 
$
and $\tau$ is small if $\mu>0$ is small. 
Inserting this change of variables in \eqref{eq_v_noteng0} we get
\begin{align*}
F(u,\mu,\tau): = & \omega_0^{(2)}-4\pi^2 J_0v(\vec{\pi})^2\ln (u+d_0)\nonumber \\
+ & \sum\limits_{n\geq 1}\omega_n^{(1)} \tau^n (u+d_0)^n \Big(\tfrac{\mu^{n-1}}{4\pi^2 J_0v(\vec{\pi})^2} 
 + \mu^n \ln (u+d_0)\Big)
+\sum\limits_{n\geq 1} \omega_n^{(2)}\mu^n \tau^n(u+d_0)^n =0.
\end{align*}
The function $F(u,\mu,\tau)$ is real-analytic for small  $|u|,$ $|\mu|,$ $|\tau|$ and by the definition of $d_0,$ 
$$
F(0,0,0)=0\quad\text{and}\quad 
\frac{\p F}{\p u}(0,0,0)=-4\pi^2 J_0v(\vec{\pi})^2/d_0 < 0.
$$
Then by the implicit function theorem in the analytical case for sufficiently small $|\mu|$ and $|\tau|$ there exists a unique $u=u(\mu,\tau)$ solving $F(u,\mu,\tau)\equiv0$ and is given by the absolutely convergent series
\begin{equation}\label{def_u_oshkormas1}
u=\sum\limits_{n,m\geq0} c_{nm}\,\mu^n\tau^m,
\end{equation}
where $\{c_{nm}\}$ are real coefficients. Since $u(0)=0,$ $c_{00}=0.$ Inserting the representation \eqref{def_u_oshkormas1} of $u$ in the expression of $\alpha$ in \eqref{alpha_d299} we get \eqref{asymp_at_v2pi0}.

Before solving \eqref{eq_v_ten01} and \eqref{eq_v_ten02} in $\alpha$ we observe that the function $\lambda:=\lambda(\cdot)$ is continuous and satisfies
$$
\lambda(0):=\lim\limits_{\alpha\searrow0} \lambda(\alpha)=0.
$$
By continuity, there exists $\alpha_1>0$ such that $\lambda(\alpha)\in(0,\mu_v^0/2)$ for all $\alpha\in(0,\alpha_1).$ For such $\alpha$ we can write
\begin{equation}\label{bir_taqsim_lambda0}
\frac{\lambda}{\mu_v^0\lambda +  (\mu_v^0)^2} = \frac{\lambda}{(\mu_v^0)^2(1 + \lambda/\mu_v^0)} = \frac{\lambda}{(\mu_v^0)^2}\sum\limits_{n\ge0} (-1)^n\, \frac{\lambda^n}{(\mu_v^0)^n}, 
\end{equation}

(b) Since $\omega_1^{(1)}>0$ in \eqref{eq_v_ten01}, setting
\begin{equation*}
\alpha:=\tau\Big(\tfrac{1}{\omega_1^{(1)}\,(\mu_v^0)^2}  + u\Big),\quad \tau:=-\frac{\lambda}{\ln \lambda},\quad \sigma:=-\frac{\ln\ln \lambda^{-1}}{\ln\lambda},
\end{equation*}
in \eqref{eq_v_ten01} and using \eqref{bir_taqsim_lambda0}, as in (a) we get \eqref{asymp_at_vpi0mas}.

(c) Since $\omega_1^{(2)}<0$ in \eqref{eq_v_ten02}, setting
\begin{equation*}
\alpha:=\lambda\Big( -\tfrac{1}{\omega_1^{(2)} (\mu_v^0)^2}  + u\Big),\quad \theta:=-\lambda\ln \lambda,
\end{equation*} 
in \eqref{eq_v_ten02} and using \eqref{bir_taqsim_lambda0}, as in (a) we get \eqref{asymp_at_vpi0}.
\end{proof}

Finally we study the threshold resonances and threshold eigenfunctions.

\begin{proposition}\label{prop:reson_rank1}
Let $v(\vec\pi)=0$ and
$$
f_v(p) = \frac{v(p)}{\emax - \dispersion(p)}.
$$
\begin{itemize}
\item[(a)] Let $\nabla v(\vec\pi)\ne0.$ Then $f_v\in L^1(\T^2)\setminus L^2(\T^2).$

\item[(b)] Let $\nabla v(\vec\pi)=0.$ Then $f_v\in L^2(\T^2).$
\end{itemize}
\end{proposition}

\begin{proof}
Since both
$v(p)^2$ and $\dispersion(q) - \emax$ behave like $(p - \vec\pi)^2$ near $\vec\pi,$ repeating a similar argument to Proposition \ref{prop:kernel_asymptotics_max} with $v=|v|$ we get
$$
\int_{\T^2} |f_v|\d p = \lim\limits_{z\to\emax} \int_{\T^2} \frac{|v(p)|\d p}{z - \dispersion(p)}<+\infty,
$$
hence, $f_v\in L^1(\T^2).$  Now if $\nabla v(\vec\pi)\ne0,$ then
$v(p)^2$ behaves like $(p-\vec\pi)^2$ and $(\dispersion(q) - \emax)^2$ behaves like $(p - \vec\pi)^4$ near $\vec\pi,$ hence,
$$
\int_{\T^2} |f_v|^2\d p = \lim\limits_{z\to\emax} \int_{\T^2} \frac{|v(p)|\d p}{z - \dispersion(p)}=+\infty,
$$
hence, $f_v\notin L^2(\T^2).$
Finally, if $\nabla v(\vec\pi)=0,$ then both $v(p)^2$ and $(\dispersion(q) - \emax)^2$ behave like $(p - \vec\pi)^4$ near $\vec\pi$ and hence,
$$
\int_{\T^2} |f_v|^2\d p = \lim\limits_{z\to\emax} \int_{\T^2} \frac{|v(p)|^2\d p}{(z - \dispersion(p))^2}<+\infty
$$
so that $f_v\in L^2(\T^2).$
\end{proof}

Now we are ready to prove the main results in case $\omega\in\{\os,\oa,\ea\}$.

\begin{proof}[Proofs of  Theorems \ref{teo:existence_eigenv}- \ref{teo:thresholds} for $\omega\in\{\os,\oa,\ea\}$]
Note that if $b\le0,$ then the perturbation $V_{b}^\omega\le0$ and hence, $\sup\, \sigma(H_b^\omega(\mu))\le \sup\sigma_\ess(H_b^\omega(\mu))=\emax,$ i.e., $\sigma(H_b^\omega(\mu))\cap (\emax,+\infty)=\emptyset.$ Hence, we assume $b>0.$
Let 
$$
v_\omega(p) =
\begin{cases}
\frac{\sqrt{b}}{2\pi}\,(\sin p_1 + \sin p_2) & \text{if $\omega=\os$},\\
\frac{\sqrt{b}}{2\pi}\,(\sin p_1 - \sin p_2) & \text{if $\omega=\oa$},\\
\frac{\sqrt{b}}{2\pi}\,(\cos p_1 - \cos p_2) & \text{if $\omega=\ea$}.
\end{cases}
$$ 
Now the existence of eigenvalues $H_b^\omega(\mu)$ follows from Proposition \ref{prop:existence_rank1}, the asymptotics of eigenvalues follows from Proposition \ref{prop:asym_rank1} and the classification of threshold eigenfunctions and resonances follows from Proposition \ref{prop:reson_rank1}, all propositions applied with $v=v_\omega.$
\end{proof}

\subsection{Discrete spectrum of \text{$H_{a,b}^\es(\mu)$}}

Since $V_{a,b}^{\es}$ is rank-two, from the min-max principle (see also  \cite[Lemma 4.4]{LKhKh:2021}) we get

\begin{proposition}\label{prop:es_eigens_minmax}
The operator $H_{a,b}^\es(\mu)$ has at most two eigenvalues outside the essential spectrum $[\dispersion_{\min},\dispersion_{\max}].$ Moreover:
\begin{itemize}
\item[(a)] if $a,b>0$  resp. $a,b<0,$ then $H_{a,b}^{\es}(\mu) $ has no discrete spectrum below resp. above the essential spectrum;

\item[(b)] if $ab<0,$ then $H_{a,b}^{\es}(\mu) $ has at most one eigenvalue on either sides of the essential spectrum.
\end{itemize}
\end{proposition}

The following lemma provides an implicit equation for eigenvalues of $H_{a,b}^\es(\mu),$ which is a simple application of Fredholm determinants theory (see e.g., \cite{S:1977_adv.math}).

\begin{lemma}\label{lem:eigen_fred}
A point $z \in \C\setminus [\emin,\emax]$ is an eigenvalue of $H_{a,b}^\es(\mu)$ with multiplicity $m$ if and only if $z$ is a zero with the multiplicity $m$ of the function
\begin{equation}\label{determin00}
\Delta_{a,b}(\mu;z)  =  \Delta_a^{(1)}(\mu;z)\Delta_b^{(2)}(\mu;z) - \mu^2ab \Delta^{(3)}(z)^2,
\end{equation}
where
\begin{align*}
& \Delta_a^{(1)}(\mu;z):= 1-\frac{a\mu}{4\pi^2}\int_{\T^2}\frac{\d q}{z-\dispersion(q)},\\
& \Delta_b^{(2)}(\mu;z) := 1 - \frac{b\mu}{4\pi^2}\int_{\T^2}\frac{(\cos q_1+\cos q_2)^2 \d q}{z-\dispersion(q)}\\
& \Delta^{(3)}(z) :=\frac{1}{4\pi^2}\,\int_{\T^2}\frac{(\cos q_1+\cos q_2 )\d q}{z-\dispersion(q)}.
\end{align*}
\end{lemma}

\noindent 
The function $\Delta_{a,b}(\mu;z)$ is called the \emph{Fredholm determinant} associated to $H_{a,b}^\es(\mu).$

\begin{remark}\label{rem:eigenvectorsss}
Considering the eigenvalue equation $H_{a,b}^\es(\mu) f= z_0f$ we observe that the eigenfunctions of $H_{a,b}^\es(\mu)$  are of the form
$$
f(p)= \frac{c_1 + c_2(\cos p_1 +\cos p_2)}{z_0 - \dispersion(p)}
$$
for some constants $c_1$ and $c_2.$ Using the equality $\Delta_{a,b}(\mu;z_0)=0$ one can readily check
\begin{itemize}
 \item if $\Delta^{(3)}(z_0)\ne0,$ then we can take $c_1=\Delta_a^{(1)}(\mu;z_0)$ and $c_2=\Delta^{(3)}(z_0);$

 \item if $\Delta^{(3)}(z_0)=0$ and $\Delta_a^{(1)}(\mu;z_0)\ne0,$ then we can take $c_1=0$ and $c_2=1;$

 \item if $\Delta^{(3)}(z_0)=0$ and $\Delta_b^{(2)}(\mu;z_0)\ne0,$ then we can take $c_1=1$ and $c_2=0;$

 \item if $\Delta_a^{(1)}(\mu;z_0) = \Delta_b^{(2)}(\mu;z_0) = \Delta^{(3)}(z_0)=0,$ then $z_0$ is the eigenvalue of multiplicity two and the corresponding eigenfunctions are
 $$
 f_1(p)= \frac{1}{z_0 - \dispersion(p)}\quad\text{and}\quad
 f_2(p)= \frac{\cos p_1 + \cos p_2}{z_0 - \dispersion(p)}.
$$
\end{itemize}
\end{remark}

Let us establish some properties of the zeros of $\Delta_{a,b}(\mu;\cdot).$ Since $\mu\mapsto \Delta_{a,b}(\mu;z)=0$ is quadratic with nonzero constant term, for any $z>\emax$ the equation $\Delta_{a,b}(\cdot;z)=0$ has at most two solutions.

\begin{proposition}[\textbf{Zeros of higher multiplicity}]\label{prop:es_zero_high_multos}
A number $z_0>\emax$ is a zero of $\Delta_{a,b}(\mu;\cdot)$ of multiplicity two if and only if
$$
\Delta_a^{(1)}(\mu;z_0)=\Delta_b^{(2)}(\mu;z_0)=\Delta^{(3)}(z_0)=0.
$$
\end{proposition}

\begin{proof}
It is enough to prove the ``only if'' implication. Assume that
\begin{equation*}
 \Delta_{a,b}(\mu; z_0) = \tfrac{\partial }{\partial z}\big|_{z=z_0}\Delta_{a,b}(\mu; z_0)=0 \quad \text{for some $z_0>\emax.$}
\end{equation*} 
 By contradiction, assume that  $\Delta^{(3)}(z_0)\ne 0.$  Then $\Delta_a^{(1)}(\mu; z_0)\Delta_b^{(2)}(\mu; z_0)=\mu^2ab\Delta^{(3)}(z_0)\ne 0,$ and hence, there exists $c\ne 0$ such that $\Delta_a^{(1)}(\mu; z_0) = c\Delta^{(3)}(z_0)$ and $\Delta_b^{(2)}(\mu; z_0) = \frac{\mu^2ab}{c}\,\Delta^{(3)}(z_0).$ Then 
\begin{align*}
0=\tfrac{\partial}{\partial z}\big|_{z=z_0}\Delta_{a,b}(\mu; z)
=& \frac{b\mu\Delta^{(3)}(z_0)}{4c\pi^2}\,  \int_{\T^2} \frac{(a\mu + c(\cos q_1+\cos q_2)^2)\d q}{(z_0-\dispersion(q))^2}\ne 0,
\end{align*}
a contradiction. Hence, $\Delta^{(3)}(z_0)=0$ and so $\Delta_a^{(1)}(\mu; z_0)\Delta_b^{(2)}(\mu;z_0)=0.$ If $\Delta_a^{(1)} (\mu;z_0)\ne0,$ then again
$$
0=\tfrac{\partial}{\partial z}\big|_{z=z_0}\Delta_{a,b}(\mu; z)
= \frac{b\mu \Delta_a^{(1)}(\mu;z_0)}{4\pi^2}\int_{\T^2}\frac{(\cos q_1+\cos q_2)^2 \d q}{(z_0-\dispersion(q))^2} \ne0,
$$
a contradiction. Similarly, one shows $\Delta_b^{(2)}(\mu;z_0)=0$.
\end{proof}

Further in Section \ref{subsec:multiple_eigen} we provide a function $\dispersion$ and numbers $a,b,\mu>0$ for which $\Delta_{a,b}(\mu;\cdot)$ has a zero of multiplicity two.

\begin{lemma}[\textbf{Number of zeros of $\Delta^{(3)}$}]\label{lem:delta3_nonzero99}

The set of zeros of the function $z\mapsto \Delta^{(3)}(z)$ in $(\emax,+\infty)$ is locally finite, i.e., cannot have a (finite) limit point.

\end{lemma}

\begin{proof}
Indeed, otherwise, by analyticity we would have
$
\Delta^{(3)}\equiv0
$
in $\C\setminus [\emin,\emax].$ For $a=b=1$ consider the operator $H_{1,1}(\mu).$ As $\Delta^{(3)}\equiv 0,$ the Fredholm determinant $\Delta_{1,1}(\mu;z)$ associated to $H_{1,1}^\es(\mu).$   factors out as
$$
\Delta_{1,1}(\mu;z)  =  \Delta_1^{(1)}(\mu;z)\Delta_1^{(2)}(\mu;z).
$$
In view of Proposition \ref{prop:existence_rank1} applied with $v\equiv1$ resp. $v =\cos p_1 + \cos p_2,$ for any $\mu>0$ the  function $\Delta_1^{(1)}(\mu;\cdot)$ resp. $\Delta_1^{(2)}(\mu;\cdot)$ has a unique zero in $(\emax,+\infty).$ Thus, by Lemma \ref{lem:eigen_fred} the operator $H_{1,1}^\es(\mu)$ has two eigenvalues for any $\mu>0.$ Then by \eqref{sigma_Hmu} the operator $H_{1,1}(\mu)$ has  at least two eigenvalues for any $\mu>0.$ However, this contradicts to \cite[Theorem 1.4]{KhLA:2021_jmaa} since for sufficiently small $\mu>0$ the operator $H_{1,1}(\mu)$ has a unique eigenvalue in $L^2(\T^2).$
\end{proof}

\begin{remark}
If $\dispersion$ is $\pi$-periodic in each coordinates, then $\Delta^{(3)}(z)\equiv0$ in $\C\setminus [-\emin,\emax].$ However, such $\dispersion(\cdot)$  do not satisfy Hypothesis \ref{hyp:main} as it has at least two maximum points in $\T^2$.
\end{remark}

\begin{lemma}[\textbf{Monotonicity of $C^1$-zeros}]\label{lem:zeros_delta_es}
Let for some $\mu_0>0$ and $\epsilon>0$ there exist
$E\in C^1([\mu_0-\epsilon,\mu_0+\epsilon])$ such that
$
\Delta_{a,b}(\mu;E(\mu))\equiv0\quad\text{in $(\mu_0-\epsilon,\mu_0+\epsilon).$}
$
Then $E'(\cdot)>0$ in $(\mu_0-\epsilon,\mu_0+\epsilon).$
\end{lemma}

\begin{proof}
By the definition \eqref{determin00} of $\Delta_{a,b}$
\begin{equation}\label{rio999}
\Delta_a^{(1)}(\mu;E(\mu))\, \Delta_b^{(2)}(\mu;E(\mu)) - \mu^2ab\, \Delta^{(3)}(E(\mu))^2\equiv0\quad\text{in $(\mu_0-\epsilon,\mu_0+\epsilon).$}
\end{equation}
Since $\mu\mapsto \Delta_{a,b}(\mu;z)$ is quadratic, $E(\mu)$ cannot be locally constant.

Assume first that $\Delta^{(3)}(E(\mu))\ne0$ for some $\mu.$ Then by \eqref{rio999} there exists $c_\mu\ne0$ such that $\Delta_a^{(1)}(\mu;E(\mu))=c_\mu\, \Delta^{(3)}(E(\mu))$ and $\Delta_b^{(2)}(\mu;E(\mu)) = \frac{\mu^2 ab}{c_\mu}\,\Delta^{(3)}(E(\mu)).$ Hence, differentiating \eqref{rio999}  in $\mu$ and simplying we get
\begin{equation}\label{Emu_hosila}
E'(\mu) \int_{\T^2}\frac{(a\mu + c_\mu(\cos q_1+\cos q_2))^2\d q}{(E(\mu) - \dispersion(q))^2} = \frac{1}{\mu} \int_{\T^2}\frac{(a\mu + c_\mu(\cos q_1+\cos q_2))^2\d q}{E(\mu) - \dispersion(q)}.
\end{equation}
Since both integrals are positive, $E'(\mu)>0.$

Now assume that $\Delta^{(3)}(E(\mu))=0.$ Since $E(\cdot)$ is not locally constant, by Lemma \ref{lem:delta3_nonzero99} the set of such $\mu$'s are locally finite. Fix any $\bar\mu$ for which $\Delta^{(3)}(E(\bar\mu))=0$ and take any sequence $\mu_k\searrow \bar\mu$ for which \eqref{Emu_hosila} holds with $\mu=\mu_k.$ Then up to a subsequence, the sequence $c_{\mu_k}$ either converges to some finite number $\bar c$ or diverges to $+\infty.$ 

In the first case, writing \eqref{Emu_hosila} with $\mu=\mu_k$ and letting $k\to+\infty$ as well as using the continuity of $E(\cdot)$  and $E'(\cdot)$ we get
$$
E'(\bar \mu) \int_{\T^2}\frac{(a\bar\mu + \bar c(\cos q_1+\cos q_2))^2\d q}{(E(\bar \mu) - \dispersion(q))^2} = \frac{1}{\bar \mu} \int_{\T^2}\frac{(a\bar\mu + \bar c(\cos q_1+\cos q_2))^2\d q}{E(\bar \mu) - \dispersion(q)}
$$
and hence, $E'(\bar\mu)>0.$

In the second case, first dividing \eqref{Emu_hosila} by $c_{\mu_k}$ and then letting $k\to+\infty$ we get
$$
E'(\bar \mu) \int_{\T^2}\frac{(\cos q_1+\cos q_2)^2\d q}{(E(\bar \mu) - \dispersion(q))^2} = \frac{1}{\bar \mu} \int_{\T^2}\frac{( \cos q_1+\cos q_2)^2\d q}{E(\bar \mu) - \dispersion(q)},
$$
hence, again $E'(\bar\mu)>0.$
\end{proof}

Now we are ready to study the existence of zeros of $\Delta_{a,b} (\mu;z)$ or equivalently the existence of eigenvalues of $H_{a,b}^\es(\mu).$

\begin{proposition}\label{prop:zeros_delta_es}
Let $\mu_\es^0$ be given as in Theorem \ref{teo:existence_eigenv}.

\begin{itemize}
 \item[(a)] Assume that $ab<0$ and $a+4b\ge0.$ Then for any $\mu>\mu_\es^0=0$ the function $\Delta_{a,b}(\mu; \cdot)$ has a unique zero in $(\emax,+\infty).$

\item[(b)] Assume that $ab<0$ and $a+4b<0.$ Then for any $\mu>\mu_\es^0>0$ the function $\Delta_{a,b}(\mu; \cdot)$ has a unique zero in  $(\emax,+\infty)$ and  for any $\mu\in(0,\mu_\es^0]$ the function $\Delta_{a,b}(\mu; \cdot)>0$ in $(\emax,+\infty).$

\item[(c)] Assume that $a,b>0$. Then for any $\mu\in(0,\mu_\es^0]$ the function $\Delta_{a,b} (\mu; \cdot)$ has a unique zero in $(\emax,+\infty)$ and for any $\mu>\mu_\es^0 $ the function $\Delta_{a,b} (\mu; \cdot)$ has two zeros in  $(\emax,+\infty).$
\end{itemize}
\end{proposition}

\begin{proof}
Applying Proposition \ref{prop:kernel_asymptotics_max} with $v\equiv 1,$ $v(q)=(\cos q_1+\cos q_2)^2$ and $v(q)=\cos q_1+\cos q_2,$ respectively, we find $\delta=\delta_\dispersion>0$ such that for any $z\in(\emax,\emax+\delta)$
\begin{align}
&\frac{1}{4\pi^2} \int_{\T^2}\frac{\d p}{z-\dispersion(p)}  = -J_0\,\ln \alpha+\ln \alpha\sum\limits_{n\geq 1} w_n^{(1)}\alpha^n+\sum\limits_{n\geq 0} w_n^{(2)}\alpha^n, \label{v_teng1}\\
&\frac{1}{4\pi^2} \int_{\T^2}\frac{(\cos p_1 +\cos p_2)^2\d p}{z-\dispersion(p)} =-4J_0\ln \alpha+\ln \alpha\sum\limits_{n\geq 1} \hat{w}_n^{(1)}\alpha^n+\sum\limits_{n\geq 0} \hat{w}_n^{(2)}\alpha^n,  \label{v_teng_cos_kvad}\\
& \frac{1}{4\pi^2} \int_{\T^2}\frac{(\cos p_1 +\cos p_2)\d p}{z-\dispersion(p)} = 2J_0\ln \alpha+\ln \alpha\sum\limits_{n\geq 1} \tilde{w}_n^{(1)}\alpha^n+\sum\limits_{n\geq 0} \tilde{w}_n^{(2)}\alpha^n, \label{v_teng_cos}
\end{align}
where $\alpha:=z-\emax$ and $\{w_n^{(1)}\},\{w_n^{(2)}\},\{\hat{w}_n^{(1)}\},\{\hat{w}_n^{(2)}\},\{\tilde{w}_n^{(1)}\},\{\tilde{w}_n^{(2)}\}$ are real numbers. Inserting these expansions in the definition of $\Delta_{a,b}$ we get
\begin{align}\label{maindetdelta_4}
\Delta_{a,b}(\mu;z)& = 1 + A_0 \mu + \tilde A_0 \mu^2+[(a+4b) - ab\tilde B_0 \mu]J_0\mu \ln \alpha \nonumber \\
& +\sum\limits_{n\geq 1}(A_n\mu+\tilde{A}_n\mu^2) \alpha^n +\sum\limits_{n\geq 1}(B_n\mu + \tilde{B}_n\mu^2) \alpha^n \ln \alpha+\sum\limits_{n\geq 1}\tilde C_n\mu^2 \alpha^n \ln^2 \alpha
\end{align}
where $A_n,\tilde{A},B_n,\tilde{B},\tilde C_n$ are  real numbers (depending only on $\dispersion.$ $a$ and $b$), $J_0$ is given by Remark \ref{rem:morse_lemma}, $\alpha :=z-\emax,$ and
\begin{subequations}
\begin{align}
A_0&:= -aw_0^{(2)} - b\hat{w}_0^{(2)}, \label{A0} \\
\tilde A_0&:= ab(w_0^{(2)}\hat{w}_0^{(2)}- [\tilde{w}_0^{(2)}]^2),\label{tildeA0} \\
\tilde B_0 &:=4w_0^{(2)}+\hat w_0^{(2)}+4\tilde w_0^{(2)}, \label{tildeB0}\\
\tilde{B}_1&:= ab\,( \hat{w}_1^{(1)} w_0^{(2)}+\hat{w}_0^{(2)} w_1^{(1)}-2\tilde{w}_1^{(1)} \tilde{w}_0^{(2)} )-ab\,J_0( 4w_1^{(2)}+\hat{w}_1^{(2)} +4\tilde{w}_1^{(2)}), \label{B1}\\
B_1&:= -a\,w_1^{(1)}-b\,\hat{w}_1^{(1)}. \label{tildeB1}\\
\tilde C_1 &:= -ab J_0 (4w_1^{(1)} + \hat w_1^{(1)} + 4\tilde w_1^{(1)} ) \label{tildeC1}
\end{align}
\end{subequations}
By \eqref{v_teng1}-\eqref{v_teng_cos}
$$
\begin{aligned}
\frac{1}{4\pi^2} \int_{\T^2}\frac{(2+\cos p_1 +\cos p_2)^2\d p}{z-\dispersion(p)} = &\ln \alpha\sum\limits_{n\geq 1}
(4w_n^{(1)} + \hat{w}_n^{(1)} + 4\tilde w_n^{(1)})
\alpha^n
+\sum\limits_{n\geq 0}(4w_n^{(2)}+ \hat{w}_n^{(2)} + 4\tilde w_n^{(2)})\alpha^n
\end{aligned}
$$
Then by the definition \eqref{def:mu_es} of $\gamma_\es$ and Corollary \ref{cor:integral_finite00} applied with $v(p)=(2+\cos p_1+\cos p_2)^2$ and using $v(\vec\pi)=0,$ $\nabla v(\vec\pi)=0$ and $\nabla^2v(\vec\pi)=0,$  we find
\begin{align*}
\tilde B_0 = 4w_0^{(2)}+\hat{w}_0^{(2)}+4\tilde{w}_0^{(2)} = \frac{1}{4\pi^2}\int_{\T^2}\frac{(2 + \cos q_1 + \cos q_2)^2\d q}{\emax - \dispersion(q)} =\frac1{\gamma_\es} ,
\end{align*}
$$
4w_1^{(1)} + \hat w_1^{(1)} + 4\tilde w_1^{(1)}  = 0
$$
and
\begin{equation}\label{mana_hosilasisis}
4w_1^{(2)} + \hat w_1^{(2)} + 4\tilde w_1^{(2)}  = -\frac{1}{4\pi^2} \int_{\T^2}\frac{(2+\cos p_1 +\cos p_2)^2\d p}{(\emax-\dispersion(p))^2}
\end{equation}
In particular, $\tilde B_0>0,$ $\tilde C_1=0,$ and the last sum in \eqref{maindetdelta_4} starts from $n=2.$
\smallskip

(a) Assume that $ab<0$ and $a+4b\ge0.$ Then for any $\mu>0$ the coefficient $[(a+4b) - ab\tilde B_0 \mu]J_0\mu $ staying in front of the leading term of the expansion \eqref{maindetdelta_4} is positive. In particular,
$$
\lim\limits_{z\searrow\emax} \Delta_{a,b}(\mu;z)  = -\infty.
$$
Since $\Delta_{a,b}(\mu;z)\to1$ as $z\to+\infty,$ by the continuity of $\Delta_{a,b}(\mu;\cdot),$ for any $\mu>0$ there exists at least one $E_\es(\mu)>\emax$ such that $\Delta_{a,b}(\mu;E_\es(\mu))=0.$ By Lemma \ref{lem:eigen_fred} $E_\es(\mu)$ is an eigenvalue of $H_{a,b}^\es(\mu).$ By Proposition \ref{prop:es_eigens_minmax} (b) $E_\es(\mu)$ is the unique eigenvalue of $H_{a,b}^\es(\mu)$ in $(\emax,+\infty),$ and therefore, again by Lemma \ref{lem:eigen_fred}  $E_\es(\mu)$ is the unique zero of $\Delta_{a,b}(\mu;\cdot)$ in $(\emax,+\infty).$
\smallskip

(b) Assume that $ab<0$ and $a+4b<0.$ Then
$
\mu_\es^0=\frac{a+4b}{\tilde B_0 ab}>0.
$
Hence, as in the proof of the assertion (a), $\Delta_{a,b}(\mu;\cdot)$ has a unique zero in $(\emax,+\infty)$ for any  $\mu\in(\mu_\es^0,+\infty).$ Let us prove that $\Delta_{a,b}(\mu;z)>0$ for any $\mu\in(0,\mu_\es^0]$ and $z>\emax.$

(b1) Assume that $\mu\in(0,\mu_\es^0).$
Since
$$
\lim\limits_{z\searrow\emax} \Delta_{a,b}(\mu;z)  = +\infty,\quad \mu\in(0,\mu_\es^0],
$$
the function $\Delta_{a,b}(\mu;\cdot)$ admits positive values near $z=\emax$ and $z=+\infty.$ Hence, if $\Delta_{a,b}(\mu;\cdot)$ has a zero, then by continuity it admits either at least two different zeros or at least one zero with multiplicity two. Then by Lemma \ref{lem:eigen_fred} $H_{a,b}^\es(\mu)$ should have at least two zeros which contradicts to Proposition \ref{prop:es_eigens_minmax} (b).

(b2) Assume that $\mu=\mu_\es^0.$ Letting $z\searrow\emax$ we get
$$
\lim\limits_{z\searrow \emax} \,\Delta_{a,b}(\mu_\es^0;z) = 1 + A_0\mu_\es^0 + \tilde A_0(\mu_\es^0)^2.
$$
Observe that
\begin{align}
1 + A_0\mu_\es^0 + \tilde A_0 (\mu_{\es}^0)^2  
= & -\frac{[(2w_0^{(2)}+\tilde{w}_0^{(2)})a-(4\tilde{w}_0^{(2)}+2\hat{w}_0^{(2)})b]^2}{ab\tilde B_0^2}. \label{constant00}
\end{align}
Applying Corollary \ref{cor:integral_finite00} with $v(p) = 2+\cos p_1+\cos p_2$ and $v(p)=2(\cos p_1 +\cos p_2)(2+\cos p_1+\cos p_2),$ and recalling the definitions \eqref{theta_mayor} and \eqref{theta_podpolkovnik} of $\Theta^*$ and $\Theta^{**}$ we establish
\begin{equation}\label{theta_star}
2w_0^{(2)} + \tilde{w}_0^{(2)} = \frac{1}{4\pi^2}\int_{\T^2} \frac{(2+\cos p_1+\cos p_2)\d p}{\emax - \dispersion(p)} = \Theta^*
\end{equation}
and
\begin{equation}\label{theta_starstar}
4\tilde{w}_0^{(2)} + 2\hat w_0^{(2)} = \frac{1}{4\pi^2}\int_{\T^2} \frac{2(\cos p_1 +\cos p_2)(2+\cos p_1+\cos p_2)\d p}{\emax - \dispersion(p)} = \Theta^{**}.
\end{equation}
Thus, recalling $\tilde B_0=1/\gamma_\es$ \eqref{constant00} reads as
\begin{equation}\label{ozod_had_hollll}
1 + A_0\mu_\es^0 + \tilde A_0 (\mu_{\es}^0)^2=
- \frac{\gamma_\es^2 [\Theta^*a - \Theta^{**}b]^2}{ab}.
\end{equation}
If $\Theta^*a \ne \Theta^{**}b,$ using $ab<0$ we get $1 + A_0\mu_\es^0 + \tilde A_0 (\mu_{\es}^0)^2>0.$ Hence, again $\Delta_{a,b}(\mu_\es^0;\cdot)$ admits positive values near $z=\emax$ and $z=+\infty.$ Hence, as in (a) $\Delta_{a,b}(\mu_\es^0;\cdot)$ has no zeros in $(\emax,+\infty).$ Finally, assume that $\Theta^*a= \Theta^{**}b.$ In this case, using \eqref{B1}, \eqref{tildeB1},   \eqref{theta_star}, \eqref{theta_starstar} and \eqref{mana_hosilasisis}
we get
\begin{align}
B_1\mu_{\es}^0+\tilde{B}_1(\mu_{\es}^0)^2 
%
=  \frac{(a+4b)(2\tilde{w}_1^{(1)}+\hat{w}_1^{(1)})[(2w_0^{(2)}+\tilde{w}_0^{(2)})a-(4\tilde{w}_0^{(2)}+2\hat{w}_0^{(2)})b]}{ab\,\tilde B_0^2}\nonumber \\
- ab\mu_{es}^2J_0 (4w_1^{(2)}+4\tilde{w}_1^{(2)}+\hat{w}_1^{(2)}) = \frac{J_0 ab (\mu_{\es}^0)^2}{4\pi^2} \int_{\T^2}\frac{(2+\cos p_1 +\cos p_2)^2\d p}{(\emax-\dispersion(p))^2}. \label{pos_consta8191}
\end{align}
Since $ab<0,$ $B_1\mu_\es^0 + \tilde B_1(\mu_\es^0)^2<0$ and hence, by \eqref{maindetdelta_4} $\frac{\Delta_{a,b}(\mu_\es^0;z)}{(z-\emax)\ln(z-\emax)}<0$ for sufficiently small $z-\emax>0.$ Thus, again $\Delta_{a,b}(\mu_\es^0;\cdot)$ is positive near $z=\emax$ and $z=+\infty$ as so  it has no zeros in $(\emax,+\infty).$
\smallskip

(c) Assume that $a,b>0.$ Then
$
\mu_\es^0=\frac{a+4b}{\tilde B_0 ab}>0.
$
Since $H_{a,b}^\es(\mu)$ has at most two zeros in $(\emax,+\infty) $ (see  Proposition \ref{prop:es_eigens_minmax} (a)), by Lemma \ref{lem:eigen_fred} $\Delta_{a,b}(\mu;\cdot)$ has at most two zeros greater than $\emax.$

(c1) Assume that $\mu\in(0,\mu_\es^0).$ Since
\begin{equation}\label{asymp_deleeee}
\lim\limits_{z\searrow\emax} \Delta_{a,b}(\mu;z)  = -\infty\quad \text{and} \quad \lim\limits_{z\to+\infty} \Delta_{a,b}(\mu;z)  = 1,
\end{equation}
by continuity, $\Delta_{a,b}(\mu;\cdot)$ has at least one zero $z_1(\mu)$ in $(\emax,+\infty).$ Assume that $\Delta_{a,b}(\mu;\cdot)$ has another zero $z_2(\mu)$ in $(\emax,+\infty).$ Then \eqref{asymp_deleeee} and easy continuity arguments show that $\Delta_{a,b}(\mu;\cdot)$ should have at least three zeros (counted with multiplicities) in $(\emax,+\infty)$ (see Figure \ref{fig:location_of_zeros} (a)) so that by  Lemma \ref{lem:eigen_fred} $H_{a,b}^\es(\mu)$ has three eigenvalues greater than $\emax.$ However, this contradicts to Proposition \ref{prop:es_eigens_minmax} (a). Thus, $\Delta_{a,b}(\mu;\cdot)$ has a unique zero.

(c2) Assume that $\mu=\mu_\es^0.$ As in the proof of (b2) we can prove that $\Delta_{a,b}(\mu_\es^0;\cdot )$ has at least one zero in $(\emax,+\infty).$ By contradiction assume that $\Delta_{a,b}(\mu_\es^0;\cdot)$ has two zeros $z_1\ge z_2>\emax.$ By Lemma \ref{lem:eigen_fred} $z_1$ and $z_2$ are two eigenvalues of  $H_{a,b}^\es(\mu_\es^0).$
Since $z_1,z_2>\emax$ and the operator-valued map $\mu\mapsto H_{a,b}^\es(\mu)$  is analytic, by the Kato-Rellich perturbation theory \cite[Chapter XII]{RS:vol.IV} there exists $\epsilon>0$ such that for any $\mu\in(\mu_\es^0-\epsilon,\mu_\es^0+\epsilon)$ the operator $H_{a,b}^\es(\mu)$ has two eigenvalues. However, this contradicts to assertion (a) since $H_{a,b}^\es(\mu)$ has a unique eigenvalue for $\mu\in(\mu_\es^0-\epsilon,\mu_\es^0).$ Thus, $\Delta_{a,b}(\mu_\es^0;\cdot)$ has a unique zero.

(c3) Assume that $\mu>\mu_\es^0.$ In this case
\begin{equation}\label{asymp_dolooo}
\lim\limits_{z\searrow\emax} \Delta_{a,b}(\mu;z)  = +\infty\quad \text{and} \quad \lim\limits_{z\to+\infty} \Delta_{a,b}(\mu;z)  = 1.
\end{equation}
By Proposition \ref{prop:existence_rank1} applied with $v\equiv \sqrt{\frac{a}{2\pi}}$ and $v(p)=\sqrt{\frac{b}{2\pi}}(\cos p_1 + \cos p_2)$ for any $\mu>0$ the functions $\Delta_a^{(1)}(\mu;\cdot)$ and $\Delta_b^{(2)}(\mu;\cdot)$ have unique zeros $z_1$ and $z_2,$ respectively. Then either
$$
\text{either}\quad \Delta^{(3)}(z_1) = \Delta^{(3)}(z_2)=0
\quad\text{or}\quad 
\Delta^{(3)}(z_1)^2 + \Delta^{(3)}(z_2)^2>0.
$$
\begin{figure}[h]
\begin{minipage}{0.45\textwidth}
\centering
\includegraphics[width=\textwidth]{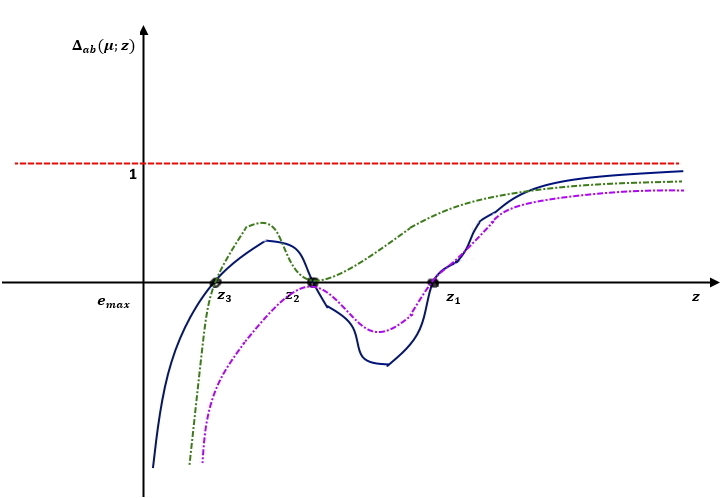}
\subcaption{(a)}
\end{minipage}
\quad
\begin{minipage}{0.45\textwidth}
\centering
\includegraphics[width=\textwidth]{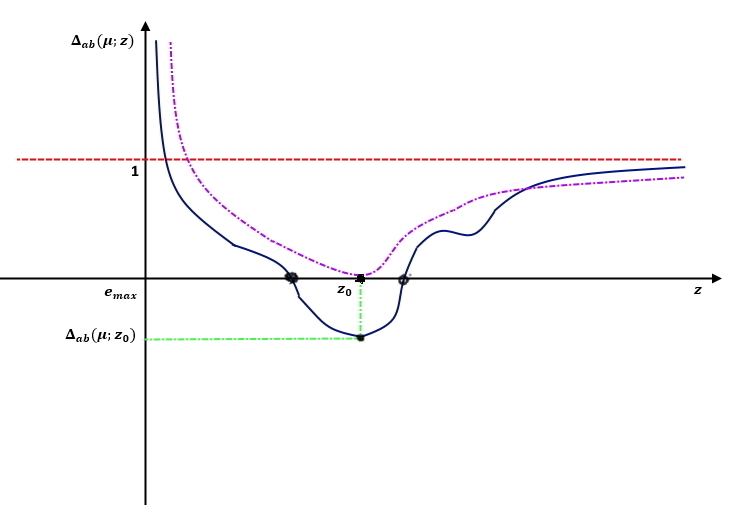}
\subcaption{(b)}
\end{minipage}
\caption{\small A schematic location of  zeros of $\Delta_{a,b}(\mu;\cdot)$ for fixed $a,b>0$ and $\mu>\mu_\es^0.$} \label{fig:location_of_zeros}
\end{figure}

In the first case both $z_1$ and $z_2$ are zeros of $\Delta_{a,b}(\mu;\cdot)$. In the second case, $\Delta_{a,b}(\mu;\cdot)$ admits negative values either at $z_0=z_1$ or $z_0=z_2.$ Now the asymptotics \eqref{asymp_dolooo} at $z=\emax$ and at $z=-\infty$ as well as the continuity of $\Delta_{a,b}(\mu;\cdot)$ imply that $\Delta_{a,b}(\cdot)$ has a zero in $(\emax,z_0)$ and in $(z_0,+\infty),$ respectively (see Figure \ref{fig:location_of_zeros} (b)). Hence, $\Delta_{a,b}(\mu;\cdot)$ has two zeros in $(\emax,+\infty).$
\end{proof}

\begin{proposition}\label{prop:analytic_eigen}
(a) Assume that $ab<0$ and for $\mu>\mu_\es^0$ let $E_\es(\mu)$ be the unique zero of $\Delta_{a,b}(\mu;\cdot)$. Then $E_\es(\cdot)$ is real-analytic in $(\mu_\es^0,+\infty).$

(b) Assume that $a,b>0$ and for $\mu\in(0,\mu_\es^0]$ let $E_\es^{(1)}(\mu)$ and for $\mu>\mu_\es^0$ let $E_\es^{(1)}(\mu)$ and $E_\es^{(2)}(\mu)$ be the zeros of $\Delta_{a,b}(\mu;\cdot).$ Then one can arrange $E_\es^{(1)}(\mu)$ and $E_\es^{(2)}(\mu)$ in such a way that $E_\es^{(1)}(\mu)$ is analytic at any $\mu>0$ and $E_\es^{(1)}(\mu)$ is analytic at any $\mu>\mu_\es^0.$
\end{proposition}

\begin{proof}
The assertion (a) follows from the Implicit Function Theorem.
Since the zeros of $\Delta_{a,b}(\mu;\cdot)$ coincide with the eigenvalues of $H_{a,b}^\es(\mu)$ (by Lemma \ref{lem:eigen_fred}),
the assertion (b) follows from the  Kato-Rellich theory (see e.g. \cite[Proposition 2.1]{KS:1980}).
\end{proof}

Note that due to higher multiplicity of zeros, the  analyticity does not allow us state in general $E_\es^{(1)}(\mu)\ge E_\es^{(2)}(\mu)$ for all $\mu>\mu_\es^0.$
Since the functions $z=E_\es(\mu),$ $z=E_\es^{(1)}(\mu)$ and $z=E_\es^{(2)}(\mu)$ solve the implicit equation $\Delta_{a,b}(\mu;z)\equiv0$ from Lemma \ref{lem:zeros_delta_es} and Proposition \ref{prop:zeros_delta_es} we get

\begin{corollary}\label{cor:increasing_eigen}
The functions $E_\es(\cdot),E_\es^{(2)}(\cdot)$ and $E_\es^{(1)}(\cdot)$ strictly increase in $(\mu_\es^0,+\infty)$ and   $(0,+\infty),$ respectively.
\end{corollary}

Now we are ready to prove Theorem \ref{teo:existence_eigenv} for $\omega=\es.$

\begin{proof}[Proof of Theorem \ref{teo:existence_eigenv} in $L^{2,\es}(\T^2)$]

The existence part of the theorem follows from Proposition \ref{prop:zeros_delta_es} while the real-analyticity and strict monotonicity of eigenvalues follow from Proposition \ref{prop:analytic_eigen} and Corollary \ref{cor:increasing_eigen}. Finally, the  convexity of $E_\es(\mu)$ follows from the obvious equality
$$
E_\es(\mu) = \sup\limits_{f\in L^{2,\es}(\T^2),\,\|f\|=1}\, (H_{a,b}^\es(\mu) f,f).
$$
\end{proof}

Now we establish  the expansions \eqref{esX000}, \eqref{esY000} and \eqref{esY001} of $E_\es(\mu),$ $E_\es^{(1)}$ and $E_\es^{(2)}$ at the corresponding thresholds.

\begin{proof}[Proof of Theorem \ref{teo:depend_eigenv_II}]

(a) Assume that either $\frac{a+4b}{ab}\le0$ or $a,b>0$  and let $X:=X_{a,b}$ be given by \eqref{def:Xab09}. By definition and Proposition \ref{prop:zeros_delta_es} $\Delta_{a,b}(\mu;X(\mu))\equiv0$ for all $\mu>0.$ Since  $X(\mu)\searrow\emax$ as $\mu\searrow 0,$ we can apply the expansion \eqref{maindetdelta_4}  of $\Delta_{a,b}(\mu;\cdot)$ with $z=X(\mu)$
\begin{align}\label{def_falpha00}
F(\mu;\alpha)& := \Delta_{a,b}(\mu;X) = 1 + A_0 \mu + \tilde A_0 \mu^2+[(a+4b) - ab\tilde B_0 \mu]J_0\mu \ln \alpha \nonumber \\
& +\sum\limits_{n\geq 1}(A_n\mu+\tilde{A}_n\mu^2) \alpha^n +\sum\limits_{n\geq 1}(B_n\mu + \tilde{B}_n\mu^2) \alpha^n \ln \alpha+\sum\limits_{n\geq 2}\tilde C_n\mu^2 \alpha^n \ln^2 \alpha
\end{align}
for all sufficiently small $\mu>0,$ where $\alpha:=X-\emax$ and we took into account $\tilde C_1=0.$  Thus, we need to solve the equation
\begin{equation}\label{implicit_eq}
F(\mu;\alpha)=0
\end{equation}
in $\alpha$ for all $\mu>0$ small enough.
\smallskip

(a1) Assume that $a+4b\ne 0$ and let
$$
\tau:=\tfrac{1}{\mu}e^{-\frac{1}{J_0(a+4b)\mu}}.
$$
Note that $\tau>0$ is small if $\mu>0$ is small. Replacing the variable $\alpha$ in \eqref{implicit_eq} with
\begin{equation}\label{repl_alpha100}
\alpha = \mu \tau (c + u),\quad c:=e^{-\frac{ab\tilde B_0 + (a+4b)A_0}{J_0(a+4b)^2}}
\end{equation}
from \eqref{implicit_eq} we obtain
$$
G(u,\mu,\tau) = 0,
$$
where
\begin{align*}
G(u,\mu,\tau) = &\frac{ab\tilde B_0 + (a+4b)A_0}{a+4b}+J_0(a+4b) \ln (u+c) + \tilde{A}_0\mu -ab\tilde B_0 J_0\mu\ln(u+c)
\nonumber \\
& +\sum\limits_{n\geq 1} (A_n+\tilde{A}_n\mu)\mu^n\tau^n (u+c)^n \\
& +\sum\limits_{n\geq 1}(B_n + \tilde{B}_n\mu) \mu^{n-1}\tau^n (u+c)^n \Big[\mu\,\ln (u+c)-\tfrac{1}{J_0(a+4b)}\Big]\\
&+\sum\limits_{n\geq 2}\tilde C_n \mu^{n-1} \tau^n (u+c)^n \Big[\mu\,\ln (u+c)-\tfrac{1}{J_0(a+4b)}\Big]^2.
\end{align*}
The function $G(u,\mu,\tau)$  is real-analytic for small  $|u|,|\mu|,|\tau|$ and satisfies $G(0,0,0)=0$ and $\frac{\p G}{\p u}(0,0,0)\neq 0.$
Thus by the Implicit Function Theorem in the analytical case for sufficiently small $|\mu|,|\tau|$ there exists a unique $u=u(\mu,\tau)$ solving $G(u,\mu,\tau)\equiv0$ and given by the absolutely convergent series
\begin{equation}\label{analitik_u_2hol}
u=\sum\limits_{n,m\geq0}c_{nm} \mu^n\tau^m
\end{equation}
where $\{c_{nm}\}$ are real coefficients.  Since $u(0,0)=0,$ $c_{00}=0.$ Inserting the representation \eqref{analitik_u_2hol} of $u$ and \eqref{A0} of $A_0$ in the expression of $\alpha$ in \eqref{repl_alpha100}  we get \eqref{esX000} in case $a+4b>0.$

(a2) Assume that $a+4b=0$ and let
$$
\tau:=\frac{1}{\mu^3} e^{\frac{1+A_0 \mu}{J_0ab \tilde B_0 \mu^2}}.
$$
Since $ab=-4b^2<0,$ $\tau>0$ is small if $\mu>0$ is small.
Replacing the variable $\alpha$ in \eqref{implicit_eq} with
\begin{equation*} 
\alpha = \mu^3 \tau (c + u),\quad c:=e^{ \frac{\tilde A_0}{ab \tilde B_0 J_0}}
\end{equation*}
and using the representation \eqref{A0} of $A_0$ in the expression of $\alpha,$ as in (a1) we get \eqref{esX000}. 
Since $a=-4b$ from \eqref{A0} and Corollary \ref{cor:integral_finite00} applied with $v(p)=4 - (\cos p_1 +\cos _2)^2$ we get
$$
A_0 = 4b w_0^{(2)} - b\hat w_0^{(2)} = b (4 w_0^{(2)} - \hat w_0^{(2)})
=\frac{b}{4\pi}\int_{\T^2} \frac{(4 - (\cos p_1 +\cos _2)^2)\d p}{\emax - \dispersion(p)}
$$
and hence, by \eqref{kappa1} $A_0 = b\kappa_1.$
Also by defintion \eqref{tildeA0} of $\tilde A_0$ we get
$
c= e^{\frac{w_0\hat w_0 - \tilde w_0^2}{J_0 \tilde B_0}},
$
and hence, $c$ depends only on $\dispersion.$

(b) Assume that $\frac{a+4b}{ab}>0$ and let $Y:=Y_{a,b}$ be defined as \eqref{def:Yab09}. Since $\Delta_{a,b}(\mu;Y(\mu))=0$ for all $\mu>\mu_\es^0$ and $Y(\mu)\searrow\emax$ as $\mu\searrow \mu_\es^0,$ we can still write \eqref{implicit_eq} with $\alpha=Y-\emax$ provided that $\mu-\mu_\es^0>0$ is small enough.
Since $\mu_\es^0 = \frac{a+4b}{ab\tilde B_0},$ expanding $F(\cdot;\alpha)$ to power series in $\mu-\mu_\es^0$ we represent \eqref{def_falpha00} as
\begin{align*}
F(\mu;\alpha)& := 1 + A_0 \mu_\es^0 + \tilde A_0 (\mu_\es^0)^2- J_0(a+4b)(\mu - \mu_\es^0) \ln \alpha \nonumber \\
& + (A_0+2\tilde{A}_0\mu_\es^0)(\mu-\mu_\es^0)+\tilde{A}_0(\mu-\mu_\es^0)^2-ab\tilde{B}_0J_0(\mu-\mu_\es^0)^2\ln \alpha  \nonumber\\
& + \sum\limits_{n\geq 1} (A_n\mu_\es^0+\tilde{A}_n (\mu_\es^0)^2) \alpha^n+\sum\limits_{n\geq 1} (B_n\mu_\es^0+\tilde{B}_n (\mu_\es^0)^2) \alpha^n\ln \alpha \nonumber \\
& + \sum\limits_{m\ge0,n\geq 1}\hat B_{mn}(\mu-\mu_\es^0)^m \alpha^n \ln \alpha 
+ \sum\limits_{m\ge0,n\geq 2}\hat C_{mn}(\mu-\mu_\es^0)^m \alpha^n \ln^2 \alpha
\end{align*}
Recalling \eqref{ozod_had_hollll}, by the definition of $\mu_\es^0$
\begin{equation}\label{der_ozod_had}
- \frac{1 + A_0\mu_\es^0 + \tilde A_0 (\mu_{\es}^0)^2}{J_0(a+4b)} =
\frac{\gamma_\es^2 [\Theta^*a - \Theta^{**}b]^2}{J_0(a+4b)ab} =:\Lambda\ge 0.
\end{equation}

(b1) Assume that $\Theta^*a \ne  \Theta^{**}b.$ Then $\Lambda>0$ and we let
$$
\lambda=\mu-\mu_\es^0\quad \text{and}\quad \tau = \frac{1}{(\mu-\mu_\es^0)^2} e^{-\frac{\Lambda}{\mu - \mu_\es^0} }.
$$
Obviously, $\tau>0$ is small if $\mu-\mu_\es^0>0$ is small. Let
\begin{equation}\label{hahaha8171}
\alpha=\lambda^2\tau (c+u),\quad c= e^{\frac{A_0+2\tilde{A}_0\mu_\es^0+ab\tilde{B}_0 J_0 \Lambda}{J_0(a+4b)}}.
\end{equation}
By the definitions \eqref{A0}, \eqref{tildeA0} and \eqref{der_ozod_had} of $A_0,$ $\tilde A_0$ and $\Lambda$ as well as equalities $\tilde B_0=1/\gamma_\es$ and $\mu_\es^0=\frac{(a+4b)\gamma_\es}{ab}$
\begin{equation}\label{c_es_hpppqqq}
\frac{A_0+2\tilde{A}_0\mu_\es^0+ab\tilde{B}_0 J_0 \Lambda}{J_0(a+4b)} = \frac{2\gamma_\es (w_0^{(2)}\hat{w}_0^{(2)} - (\tilde{w}_0^{(2)})^2)}{J_0} - \frac{aw_0^{(2)} + b\hat{w}_0^{(2)}}{J_0(a+4b)}
+\frac{\gamma_\es [\Theta^*a - \Theta^{**}b]^2}{J_0(a+4b)^2}.
\end{equation}
Now \eqref{esY000} can be obtained as in (a1), using \eqref{hahaha8171} and \eqref{c_es_hpppqqq} in \eqref{implicit_eq}.
\smallskip

(b2) Assume that $\Theta^*a =  \Theta^{**}b.$ Then by \eqref{ozod_had_hollll}  $1 + A_0\mu_\es^0 + \tilde A_0 (\mu_{\es}^0)^2=0$ and by \eqref{pos_consta8191} $B_1\mu_{\es}^0+\tilde{B}_1(\mu_{\es}^0)^2\ne0.$
Let
$$
\lambda:=\mu-\mu_\es^0,
\quad
\tau:=-\frac{1}{\ln \lambda},
\quad
\theta=-\lambda\ln \lambda
$$
Obviously, $\tau,\theta>0$ are small if $\lambda=\mu-\mu_\es^0>0$ is small. Now using 
\begin{equation*} 
\alpha:=\lambda (u+c),
\end{equation*}
where
$$
c:=\frac{J_0(a+4b)}{B_1\mu_{\es}^0+\tilde{B}_1(\mu_{\es}^0)^2}
=\frac{4\pi^2ab}{(a+4b)\gamma_{es}^2} \Big(\int_{\T^2}\frac{(2+\cos p_1 +\cos p_2)^2\d p}{(\emax-\dispersion(p))^2}\Big)^{-1},
$$
in \eqref{implicit_eq} and repeating similar arguments as in (a1) 
we get \eqref{esY001}.
\end{proof}

Finally we study the threshold eigenvalues and threshold resonances of $H_{a,b}^\es(\mu).$

\begin{proof}[Proof of Theorem \ref{teo:thresholds} for $\omega=\es$]
Let us consider the eigenvalue equation \eqref{eigen_equation} for $\omega=\es$ in $L^{1,\es}(\T^2)$.  Since we are interested in positive threshold $\mu_\es^0>0$ we have  $\frac{a+4b}{ab}>0.$ By the definition of $H_{a,b}^\es(\mu),$ any $L^{1}$-solution $f$ of
 \eqref{eigen_equation}
must be of the form
\begin{equation}\label{equ_reso}
f(p) = C\,\frac{2+\cos p_1+\cos p_2}{\emax - \dispersion(p)}
\end{equation}
(for some constant $C\ne0$). Then by Corollary \ref{cor:integral_finite00} $f\in L^{2,\es}(\T^2).$ Let us show that this necessarily implies $\Theta^*a = \Theta^{**}b.$ Indeed, by the definitions \eqref{theta_mayor} and \eqref{theta_podpolkovnik}  we can represent \eqref{eigen_equation} as
$$
f(p) = \frac{\Theta^*a + \frac12\,\Theta^{**}b (\cos p_1+\cos p_2)}{\emax - \dispersion(p)}.
$$
Comparing this with \eqref{equ_reso} we get $\Theta^*a = \Theta^{**}b.$
\end{proof}

\section{Some examples}\label{sec:examples}

\subsection{Discrete Laplacian}\label{subsec:discrete_laplace}
Let  $\hat H_0= -\hat\Delta$ so that
$$
\dispersion(q)=2-\cos q_1-\cos q_2.
$$
Using the results in \cite[Appendix]{LKhKh:2021} we can compute all constants in \eqref{def:mu_os}-\eqref{def:mu_es}:
$$
\gamma_{\os}=\gamma_{\oa}=\frac{\pi}{2\pi -4},\quad \gamma_{\ea}=\frac{\pi}{8-2\pi},
\quad \gamma_{\es}=\frac12 =0.5.
$$
Hence, the set of possible coupling constants thresholds is $\{0,\frac{\pi}{(2\pi - 4)b},\frac{\pi}{(2-\pi)b}, \frac{a+4b}{2ab}\}.$ Moreover the constants in  \eqref{theta_mayor}-\eqref{kappa1} are also computed:
$$
\Theta^*=1,\quad \Theta^{**}=0,\quad \kappa_0=\kappa_1=2.
$$
In particular, $\Theta^*a\ne\Theta^{**}b,$ and thus, the eigenvalues of $H_{a,b}(\mu)\big|_{L^{2,\es}(\T^2)}$ releases neither from threshold resonance nor from threshold eigenvalue.
Since
$$
\int_{\T^2}\frac{(\cos q_1 + \cos q_2)\d q}{z- (2-\cos q_1 - \cos q_2)}>0
$$
(see e.g., \cite[page 9]{HMK:2020}), all zeros of the Fredholm determinant $\Delta^{\es}(\mu;\cdot)$ are simple, and hence for the functions  $E_\es^{(1)}(\cdot)$ and $E_\es^{(2)}(\cdot)$ given by Theorem \ref{teo:existence_eigenv} (b) we have $E_\es^{(1)}(\cdot)>E_\es^{(2)}(\cdot)$ for any $\mu>\mu_\es.$

\subsection{Eigenvalues of \text{$H_{a,b}^\es(\mu)$} with higher multiplicity} \label{subsec:multiple_eigen}

Given $A\in[0,1],$  let
$$
\phi_A(t)=
\begin{cases}
\cos 2t & \text{if $t\in(-\pi,-\frac{3\pi}{4}]\cup [\frac{3\pi}{4},\pi],$}\\
0 & \text{if $t\in [-\frac{3\pi}{4},-\frac{\pi}{2}]\cup [\frac{\pi}{2},\frac{3\pi}{4}],$}\\
A\cos t & \text{if $t\in[-\frac{\pi}{2},\frac{\pi}{2}]$}
\end{cases}
$$
and
$$
\dispersion_A(p) = \phi_{A}(p_1) + \phi_{A}(p_2).
$$
Notice that $\dispersion$ is a Lipschitz function on $\T^2$ and as in Example  \ref{ex:stran_epsilon} $\hat\dispersion\in\ell^1(\Z^2).$ Moreover, $\emax=1,$ $\emin =0$ and if $A\in[0,1),$ then $\vec\pi$ is the non-degenerate unique maximum point of $\dispersion(\cdot)$.  We claim that for any $z>1$ there exists $A\in(0,1)$ such that
$$
G_z(A):=\int_{\T^2}\frac{(\cos p_1 + \cos p_2)\d p_1\d p_2}{z-\dispersion_A(p)} =0.
$$
Indeed, since $\dispersion_A$ is symmetric, and even in each coordinates,
$$
G_z(A)=8\int_0^\pi\int_0^\pi \frac{\cos p_1\,\d p_1\d p_2}{z-\dispersion_A(p)} =0.
$$
Setting $E:=z-\phi_A(p_2)$ for $p_2\in(0,\pi)$ let us compute
$$
f_A(p_2):=\int_0^\pi \frac{\cos p_1\d p_1}{E-\phi_A(p_1)}=\int_0^{\pi/2} \frac{\cos p_1\d p_1}{E-A\cos p_1}
+\int_{\pi/2}^{\pi} \frac{\cos p_1\d p_1}{E - \chi_{[\frac{3\pi}{4},\pi]}\cos 2p_1}.
$$
Letting $p_1:=\pi-p_1$ in the second integral in the last expression we get
\begin{align*}
f_A(p_2)  
=& \int_0^{\pi/2} \frac{(A\cos p_1 - \chi_{[0,\frac{\pi}{4}]} \cos 2p_1)\cos p_1\d p_1}{(E-A\cos p_1)(E - \chi_{[0,\frac{\pi}{4}]}\cos 2p_1)}.
\end{align*}
Since $\cos p_1 - \chi_{[0,\frac{\pi}{4}]} \cos 2p_1>0$ and $- \chi_{[0,\frac{\pi}{4}]} \cos 2p_1\le 0$ in $(0,\pi/2),$ one has
$f_A(p_2)>0$ for $A=1$ and $f_A(p_2)<0$ for $A=0$ for all $p_2.$
Therefore, $G_z(A)>0$ for $A=1$ and $G_z(A) <0$ for $z=0.$ Since $A\mapsto G_z(A)$ is continuous, there exists $A\in(0,1)$ such that $G_z(A)=0.$

Fix $\mu>0,$ $z_0>1$ and let $A_0:=A_{z_0}\in(0,1)$ be such that $G_{z_0}(A_0)=0.$ Let
$$
a_0:= \frac{4\pi^2}{\mu}\Big(\int_{\T^2}\frac{\d q}{z_0-\dispersion_{A_0}(q)}\Big)^{-1}\quad\text{and}\quad
b_0:= \frac{4\pi^2}{\mu}\Big(\int_{\T^2}\frac{(\cos q_1+\cos q_2)^2\d q}{z_0-\dispersion_{A_0}(q)}\Big)^{-1}.
$$
For this $a_0$ and $b_0$  the terms of Fredholm determinant $\Delta(\mu;z)$  at $z=z_0$ satisfy
$$
1-\frac{a_0\mu}{4\pi^2}\int_{\T^2}\frac{\d q}{z_0-\dispersion_{A_0}(q)}=1-\frac{b_0\mu}{4\pi^2}\int_{\T^2}\frac{(\cos q_1+\cos q_2)^2 \d q}{z_0-\dispersion_{A_0}(q)}= \int_{\T^2} \frac{(\cos q_1+\cos q_2)\d q}{z_0-\dispersion_{A_0}(q)} = 0.
$$
Then by Proposition \ref{prop:es_zero_high_multos}
$$
\Delta(\mu;z_0) =\tfrac{\partial}{\partial z}\Delta^{\es}(\mu;z_0)\Big|_{z=z_0} =0,
$$
i.e., $z_0$ is the zero of $\Delta^\es(\mu;\cdot)$ of multiplicity two. Then by Lemma \ref{lem:eigen_fred} $z_0$ is the eigenvalue of $H_{a,b}(\mu)$ of multiplicity two.

\appendix

\section{Asymptotic expansion of some integrals}

In this section we find an asymptotics of integrals of the form
$$
B(z)=\int_{\T^2} \frac{v(q)\d q}{z- \dispersion(q)}
$$
near $z=\emax,$ where $v(\cdot)$ analytic function and $z\in \C/[\emin,\emax].$

\begin{proposition}[\textbf{Expansion of $B(z)$around $z=\dispersion_{\max}$}]
\label{prop:kernel_asymptotics_max}
There exists $\delta := \delta(v,\dispersion)\in(0,1)$  such that for any $z\in(\dispersion_{\max},\dispersion_{\max}+\delta)$  the function  $B(z)$ is represented as
\begin{align}\label{e_d2max}
B(z) =& - \pi J(\psi(0))\,v(\vec{\pi})\ln(z-\dispersion_{\max}) \nonumber \\
&+\ln(z-\dispersion_{\max}) \sum\limits_{n\ge1}  w_n^{(1)}(z - \dispersion_{\max})^n
+\sum\limits_{n\ge0}  w_n^{(2)}(z - \dispersion_{\max})^n,
\end{align}
where $\psi$ is given by Remark \ref{rem:morse_lemma}, $\{w_n^{(1)}\},\{w_n^{(2)}\}$ are real numbers and both power series converge absolutely.
\end{proposition}

\begin{proof}
Let $U(\vec{\pi})$ be the neighborhood of $\vec{\pi}$ given by Remark \ref{rem:morse_lemma}. Then given $z>\dispersion_{\max}$,

\begin{align}\label{B_ni_ajrat}
B(z)= & \int_{U(\vec{\pi})} \frac{ v(q) \d q}{z-\dispersion(q)}
+\int_{\T^2\setminus U(\vec{\pi})} \frac{ v(q)\, \d q}{z-\dispersion(q)}=:I_1(z)+I_2(z).
\end{align}
Since $\vec{\pi}$ is the unique maximum of $\dispersion,$
$I_2(\cdot)$ is an analytic function of $z$ in a (complex)
neighborhood of $\dispersion_{\max}$
so that
\begin{equation}\label{I_2ning_analitikligi}
I_2(z)=
\sum\limits_{k\ge0} \,(-1)^k
\int_{\T^2\setminus U(\vec{\pi})}
\frac{ v(q) \d q}{(\emax-\dispersion(q))^{k+1}}\,
(z - \dispersion_{\max})^k.
\end{equation}
Its radius of convergence $r_1$ depends only on $\dispersion$ (through $U(\vec{\pi})$) and $v.$

In $I_1(z)$ we first make  the change of variables
$q\mapsto \psi(y)$ and then use the coarea formula 
to get
\begin{align}\label{I_1ning_ajralishi}
I_1(z) = &\int_{B_\gamma(0)}
\frac{v(\psi(y))J(\psi(y)) \d y}{y^2 - \dispersion_{\max}+z} = \int_0^\gamma \frac{\d r}{r^2+z-\dispersion_{\max}}
\int_{\p B_r(0)}v(\psi(y))J(\psi(y))\d\cH^{1}(y),
\end{align}
where $\p B_r(0)\subset\R^2$ is the circle of radius $r>0$ centred at the origin and $\cH^{1}$ is the one dimensional Hausdorff measure. By the Pizzetti formula (see e.g. \cite{Olevskii:1989}),
\begin{equation}\label{adaff}
\frac{1}{2\pi r} \int_{\p B_r(0)}v(\psi(y))J(\psi(y))\d\cH^{1}(y) = \sum\limits_{n\ge0} C_n r^{2n},
\end{equation}
where  
\begin{equation}\label{koefisentlar}
C_n:=\frac{1}{4^n(n!)^2\,}\, \Delta_y^{n}[v(\psi(y))J(\psi(y))]\Big|_{y=0},\quad n\ge0,
\end{equation}
and $\Delta_y$ is the Laplace operator in $y\in\R^2.$ Since  $v,$ $\psi$ and $J$ are analytic functions, taking $\gamma>0$ smaller if necessary, we assume that the series in  \eqref{adaff} converges uniformly in $r\in[0,\gamma].$  Note that by \eqref{koefisentlar} for $n=0$
\begin{equation}\label{C_0pq}
C_0:=v(\psi(0))J(\psi(0))
=v(\vec{\pi})J(\psi(0)).
\end{equation}
Inserting  \eqref{adaff} in \eqref{I_1ning_ajralishi} we obtain
\begin{equation*}
I_1(z) = 2\pi\, \sum\limits_{n\ge0} C_n \int_0^\gamma \frac{r^{2n+1}\d r}{r^2+z-\dispersion_{\max}},
\end{equation*}
where the series is uniformly convergent in $r\in[0,\gamma].$
By \cite[Lemma B.1]{LKh:2011_jpa}, for any integer $n\ge0$ and $s\in(0,\gamma/2),$
\begin{align*}
\int_0^\gamma \frac{r^{2n+1}\d r}{r^2 + s} =
-\frac{1}2\,(-s)^n\ln s +\tilde I_n(s)
\end{align*}
where  $\tilde I_n$ are real-analytic function in $W_1:=\{z\in\C:\,|z|<1\}.$
Therefore,
$$
I_1(z) = -\pi\,\ln (z-\dispersion_{\max})\sum\limits_{n\ge0} C_n\,(\dispersion_{\max}-z)^n
+\sum\limits_{n\ge0} C_n\,\tilde I_n(z-\dispersion_{\max}).
$$
 Combining this  representation of $I_1(z)$ with the expansion \eqref{I_2ning_analitikligi} of $I_2(z)$ and taking into account \eqref{C_0pq} and \eqref{B_ni_ajrat} we obtain \eqref{e_d2max}  with $\delta:= \min\{\gamma/2,r_1/2\}\in(0,1/2).$
\end{proof}

\begin{corollary}\label{cor:integral_finite00}
Let $v$ be a real analytic function on $\T^2$ and let $\{\omega_n^{(1)}\}$ and $\{\omega_n^{(2)}\}$ be the coefficients of $B(\cdot)$ given by \eqref{e_d2max}.

\begin{itemize}
\item[(a)] If $v(\vec\pi)=0,$ then
$$
\omega_0^{(2)}=\int_{\T^2} \frac{v(q)\d q}{ \emax-\dispersion(q) }.
$$

\item[(b)] If $v(\vec\pi)=0$ and $\nabla v(\vec\pi)=0$, then 
$$
\omega_1^{(1)}=\frac{\pi  J(\psi(0))}{4}\Big[\frac{\p^2 v}{\p q_1^2}(\vec{\pi})\, \Big(\frac{\p \psi}{\p y_1}(\vec{0})\Big)^2+\frac{\p^2 v}{\p q_2^2}(\vec{\pi})\, \Big(\frac{\p \psi}{\p y_2}(\vec{0})\Big)^2\Big].
$$

\item[(c)] If $v(\vec\pi)=0$, $\nabla v(\vec\pi)=0$ and $\nabla^2 v(\vec\pi)=0,$ then
$$
\omega_1^{(1)}=0,\quad \omega_1^{(2)}=-\int_{\T^2} \frac{v(q)\,\d q}{( \emax-\dispersion(q) )^2}.
$$
\end{itemize}

\end{corollary}

\begin{proof}
\text{(a)} Let  $v(\vec\pi)=0,$ then by Proposition \ref{prop:kernel_asymptotics_max} the limit
$$
\lim\limits_{z\searrow \emax} \int_{\T^2} \frac{v(q)\d q}{z -\dispersion(q)}
$$
exists and equal to $\omega_0^{(2)}$. On the other hand, since both
$v(q)$ and $\dispersion(q) - \emax$ behave like $(q - \vec\pi)$ near $\vec\pi,$ by similar arguments to Proposition \ref{prop:kernel_asymptotics_max} (with $v=|v|$)
we find $\frac{|v|}{\emax - \dispersion(\cdot)}\in L^1(\T^2)$ and hence, by the dominated convergence theorem
$$
\int_{\T^2} \frac{v(q)}{\emax -\dispersion(q)}\d q = \lim\limits_{z\to\emax} \int_{\T^2} \frac{|v(q)|\d q}{z - \dispersion(q)} = \omega_0^{(2)}.
$$

(b)  If $v(\vec\pi)=0$ and $\nabla v(\vec\pi)=0,$
then since $\psi(0)=\vec\pi$ by \eqref{e_d2max} and \eqref{koefisentlar}
$$ 
\begin{aligned}
\omega_1^{(1)} = \pi C_1 & =\frac{\pi}{4} \Delta_y^1 v(\psi(y)) J(\psi(y))\Big|_{y=0}\\
& = \frac{\pi \, J(\psi(0))}{4} \Big[\frac{\p^2 v}{\p q_1^2}(\vec{\pi})\, \Big(\frac{\p \psi}{\p y_1}(\vec{0})\Big)^2+\frac{\p^2 v}{\p q_2^2}(\vec{\pi})\, \Big(\frac{\p \psi}{\p y_2}(\vec{0})\Big)^2\Big].
\end{aligned}
$$

(c) If $v(\vec\pi)=0,$ $\nabla v(\vec\pi)=0$ and $\nabla^2 v(\vec\pi)=0,$ then by assertion (b)  $\omega_1^{(1)}=0.$ Then the first sum in \eqref{e_d2max}  starts from $n=2$ and hence, by the absolute convergence of the power series $B(\cdot)$ is (one-sided) differentiable at $z=\emax$ and
$$
\omega_1^{(2)} = \lim\limits_{z\searrow \emax}\,B'(z) = - \lim\limits_{z\searrow \emax}\,\int_{\T^2}\frac{v(q)\d q}{(z-\dispersion(q))^2}.
$$
On the other hand, since both $v(q)$ and $(\emax-\dispersion(q))^2$ behave as $|q-\vec\pi|^4$ near $q=\vec\pi,$ as in the proof of (a) we can show $\frac{v}{(\emax - \dispersion(\cdot))^2}\in L^1(\T^2).$ Thus, by the Dominated Convergence Theorem
$$
\omega_1^{(2)} = - \lim\limits_{z\searrow \emax}\,\int_{\T^2}\frac{v(q)\d q}{(z-\dispersion(q))^2} = -\int_{\T^2} \frac{v(q)\,\d q}{ (\emax-\dispersion(q))^2 }.
$$
\end{proof}

\end{document}